\documentclass{article}

\usepackage{amsmath,amssymb,amsthm,mathrsfs}
\usepackage{hyperref}

\newtheorem{theorem}{Theorem}
\newtheorem{proposition}[theorem]{Proposition}
\newtheorem{lemma}[theorem]{Lemma}

\theoremstyle{definition}\newtheorem{example}[theorem]{Example}
\theoremstyle{definition}\newtheorem{definition}[theorem]{Definition}
\theoremstyle{definition}\newtheorem{remark}[theorem]{Remark}

\DeclareMathOperator*{\argmin}{argmin}
\DeclareMathOperator*{\argmax}{argmax}
\DeclareMathOperator*{\wlimsup}{\mathrm{w} \text{-} \! \limsup}

\DeclareMathOperator*{\wliminf}{\mathrm{w}\text{-} \!\liminf}

\DeclareMathOperator*{\glim}{\mathrm{G}\text{-} \!\lim}
\DeclareMathOperator*{\elim}{\mathrm{E}\text{-} \!\lim}
\DeclareMathOperator*{\PKlim}{\mathrm{PK}\text{-} \!\lim}
\DeclareMathOperator*{\Mlim}{\mathrm{M}\text{-} \!\lim}

\newcommand{\fonction}[5]{\begin{array}[t]{lrcl}#1 :&#2 &\rightarrow &#3\\&#4& \mapsto &#5 \end{array}}
\newcommand{\fonctionset}[5]{\begin{array}[t]{lrcl}#1 :&#2 &\rightrightarrows &#3\\&#4& \mapsto &#5 \end{array}}

\def\R{\mathbb{R}}
\def\N{\mathbb{N}}
\def\H{\mathrm{H}}
\def\I{\mathrm{Id}}

\def\C{\mathrm{C}}

\def\D{\mathrm{Dom}}

\def\E{\mathcal{E}}

\def\t{\tau}
\def\d{\mathrm{d}}

\def\dom{\mathrm{dom}}
\def\prox{\mathrm{prox}}
\def\Prox{\mathrm{Prox}}
\def\proj{\mathrm{proj}}

\def\int{\mathrm{int}}

\def\epi{\mathrm{Epi}}

\def\Gr{\mathrm{Gr}}

\usepackage{tikz}
\usetikzlibrary{decorations.pathreplacing}
\usetikzlibrary{decorations.pathmorphing}
\usetikzlibrary{arrows,shapes}
\usepackage{pstricks,pst-grad}

\newrgbcolor{white}{1. 1. 1.}
\newrgbcolor{lightblue}{0. 0. 0.80}
\newrgbcolor{whiteblue}{.80 .80 1.}
\newrgbcolor{lightred}{.85 0. 0.}
\newrgbcolor{lightred2}{.85 0.40 0.40}
\newrgbcolor{whitered}{0.85 0.45 0.45}
\newrgbcolor{lightgreen}{0. 0.50 0.}
\newrgbcolor{whitegreen}{0.60 0.95 0.40}
\newrgbcolor{lightbrown}{0.90 0.35 0.05}
\newrgbcolor{whitebrown}{0.95 0.55 0.20}
\definecolor{bovert}{RGB}{0,154,0}

\title{Sensitivity analysis of variational inequalities via twice epi-differentiability and proto-differentiability of the proximity operator}

\author{Samir Adly\footnote{Institut de recherche XLIM. UMR CNRS 7252. Universit\'e de Limoges, France. \texttt{samir.adly@unilim.fr}},
Lo\"ic Bourdin\footnote{Institut de recherche XLIM. UMR CNRS 7252. Universit\'e de Limoges, France. \texttt{loic.bourdin@unilim.fr}}, 
}

\begin{document}

\maketitle

\begin{abstract}
In this paper we investigate the sensitivity analysis of parameterized nonlinear variational inequalities of second kind in a Hilbert space. The challenge of the present work is to take into account a perturbation on all the data of the problem. This requires special adjustments in the definitions of the generalized first- and second-order differentiations of the involved operators and functions. Precisely, we extend the notions, introduced and thoroughly studied by R.T.~Rockafellar, of twice epi-differentiability and proto-differentiability to the case of a parameterized lower semi-continuous convex function and its subdifferential respectively. The link between these two notions is tied to Attouch's theorem and to the new concept, introduced in this paper, of convergent supporting hyperplanes. The previous tools allow us to derive an exact formula of the proto-derivative of the generalized proximity operator associated to a parameterized variational inequality, and deduce the differentiability of the associated solution with respect to the parameter. Furthermore, the derivative is shown to be the solution of a new variational inequality involving semi- and second epi-derivatives of the data. An application is given to parameterized convex optimization problems involving the sum of two convex functions (one of them being smooth). The case of smooth convex optimization problems with inequality constraints is discussed in details. This approach seems to be new in the literature and open several perspectives towards theoretical and computational issues in nonlinear optimization.
\end{abstract}


\medskip


\tableofcontents

\section{Introduction}
In the whole paper let $\H$ be a real Hilbert space and let $\langle \cdot , \cdot \rangle$ (resp. $\Vert \cdot \Vert$) be the corresponding scalar product (resp. norm). In the sequel, $\R_+$ denotes the set of all nonnegative real numbers.

\subsection{Context of the paper}
One of the most famous optimization problem consists in finding a point $y$ of a set $\C \subset \H$ that minimizes the distance to a given point $x \in \H$. This historical problem can be rewritten as
$$ y \in \argmin_{z \in \C} \Vert z - x \Vert. $$
If $\C$ is a nonempty closed convex set, it is well-known that the above problem admits a unique solution, denoted by $y = \proj_\C (x)$, called the \textit{metric projection} of $x$ onto $\C$, and characterized by the following \textit{variational inequality}
$$ \forall z \in \C, \quad \langle y,z-y \rangle \geq \langle x,z-y \rangle . $$
As a first step towards the \textit{sensitivity analysis} of the above variational inequality, one would consider a slight perturbation of the point $x\in \H$, replacing it by $x(t)\in \H$ where $t\in\R_+$ is a parameter. The behavior of $y(t) = \proj_\C(x(t))$ was well studied in the literature and is naturally connected to the differentiability properties (in some general sense) of the projection operator $ \proj_\C$. We mention for instance the work of E.H.~Zarantonello~\cite{Zaran} who had shown the existence of directional derivatives for $\proj_\C$ at the boundary of~$\C$. However, it is well-known that outside the set $\C$, the directional differentiability of $\proj_\C$ is not guaranteed (see \cite{shap} for a counter-example in~$\R^2$). On the other hand, A. Haraux~\cite{haraux} and F.~Mignot~\cite{Mig76} published two fundamental papers where the conical differentiability of the projection operator is proved under a \textit{polyhedric} (also called \textit{density}) assumption. In that case, an asymptotic development of $y(t)$ at $t=0$ can be obtained and the derivative $y'(0)$ can be expressed as the projection of $x'(0)$ onto a closed subset of the tangent cone of $\C$ at $y(0)$. We also refer to \cite{FitzPhelps,JBHU2,shap2} for other references concerning the differentiability of the metric projection operator.

\medskip

The theory of variational inequalities, introduced by G.~Stampacchia and G.~Fichera with motivations in unilateral mechanics, has been developed in the 1970's by the French and Italian schools. This theory has been proved to be a useful and effective tool for the study of a wide range of applications, within a unified framework, in many scientific areas such as contact mechanics, mathematical programming, operation research, economy, transportation planning, game theory, etc. In this paper we will focus on a general \textit{nonlinear variational inequality of second kind} which consists in finding $y \in \H$ such that

\vspace{-0.19cm}\noindent
\begin{minipage}{.16\linewidth}
	\begin{equation*}
	\!\!\! (\mathrm{VI}(A,f,x))
	\end{equation*}
\end{minipage}
\begin{minipage}{.8\linewidth}
	\begin{equation*}
	\forall z \in \H, \quad \langle A(y) , z - y \rangle + f(z) - f(y) \geq \langle x , z - y \rangle ,
	\end{equation*}
\end{minipage}\vspace{0.35cm}

where $A: \H \to \H$ is a (possibly nonlinear) operator, $f : \H \to \R \cup \{ +\infty \}$ is a proper lower semi-continuous convex function and $x \in \H$ is given. Recall that $(\mathrm{VI}(A,f,x))$ is said to be \textit{linear} if $A$ is linear, and that $(\mathrm{VI}(A,f,x))$ is said to be of \textit{first kind} if $f = \delta_\C$ coincides with the indicator function of a nonempty closed convex subset~$\C \subset \H$. In the case where $A = \I$ is the identity operator, $(\mathrm{VI}(A,f,x))$ admits a unique solution given by $y = \prox_{f} (x)$, where $\prox_f$ denotes the classical proximity (also well-known as proximal) operator introduced by J.-J.~Moreau~\cite{moreau65} in 1965. Note that if moreover $f = \delta_\C$, then $\prox_f = \proj_\C$. As a consequence, the Moreau's proximity operator can be viewed as a generalization of the projection operator. On the other hand, in the general case of a possibly nonlinear operator $A$, assumed to be Lipschitz continuous and strongly monotone, it can be proved \cite[Proposition~31 p.140]{Brezisarticle} that $(\mathrm{VI}(A,f,x))$ admits a unique solution denoted by $y = \prox_{A,f} (x)$, where $\prox_{A,f}$ can be seen as a generalization of the Moreau's proximity operator $\prox_f$.

\medskip

A first step towards the sensitivity analysis of $(\mathrm{VI}(A,f,x))$ is to consider a slight perturbation of the point $x\in \H$, replacing it by $x(t)\in \H$ where $t\in\R_+$ is a parameter, and to study the behavior of $y(t) = \prox_{A,f}(x(t))$. The key point in the understanding of the differentiability properties of the generalized proximity operator $\prox_{A,f}$ is to investigate generalized second-order differentiation theory. Precisely, R.T.~Rockafellar introduced and thoroughly studied the notions of \textit{twice epi-differentiability}~\cite{Rockepidiff} and \textit{proto-differentiability}~\cite{Rockproto} of functions and set-valued maps respectively. These notions are both based on Painlev\'e-Kuratowski convergence of epigraphs and graphs of difference quotients. In the particular case where $A = \I$ is the identity operator and where $\H$ is finite-dimensional, thanks to Attouch's theorem \cite{Attouch}, R.T.~Rockafellar~\cite{Rocksecond} proved the equivalence between the twice epi-differentiability of $f$, the proto-differentiability of the subdifferential operator $\partial f$ and the proto-differentiability of the Moreau's proximity operator $\prox_f$ (see also \cite[Chapter~13]{RockWets}). Note that C.N.~Do~\cite{Chi} extended these results to the infinite-dimensional Hilbert setting using the concept of Mosco convergence \cite{Mosco2}. Finally, from the twice epi-differentiability of $f$, it can be derived that $y(t) = \prox_f(x(t))$ is differentiable at $t=0$.  In the general case of a possibly nonlinear operator $A$, the additional assumption of semi-differentiability of $A$ (notion introduced by J.-P.~Penot \cite{Penotarticle}) allows to recover easily the differentiability of $y(t) = \prox_{A,f}(x(t))$ at $t=0$. Moreover $y'(0)$ can be expressed as the image of $x'(0)$ under a generalized proximity operator involving the semi-derivative of $A$ and the second epi-derivative of~$f$. In particular, note that $y'(0)$ is then the unique solution of the associated variational inequality.

\medskip

For the calculus rules of epi-derivatives, we refer to the work of R.A.~Poliquin and R.T.~Rockafellar \cite{PolRock2,PolRock} and references therein. In particular, it can be proved that Mignot and Haraux's \textit{polyhedric} assumption on a closed convex set $\C$ is a sufficient condition for the twice epi-differentiability of the corresponding indicator function $\delta_\mathrm{C}$. We refer to the discussion in \cite[Definition~2.8 and Example~2.10]{Chi}. As a consequence, Rockafellar's result encompasses the one by Mignot and Haraux. Before concluding this section, we think that it is important to mention that there exist other approaches in the literature dealing with generalized second-order differentiations. It is not our aim to give here a complete list of references. Nevertheless, to mention just a few, we can refer for example to the work of B.S.~Mordukhovich and R.T.~Rockafellar \cite{MorduRock} based on the dual-type constructions generated by coderivatives of first-order subdifferentials and to the paper of J.B.~Hiriart-Urruty and A.~Seeger \cite{JBHUSeeger} where calculus rules were introduced for set-valued second-order derivatives. 

\subsection{Contributions of the paper}
The challenge of this paper is to investigate the sensitivity analysis of $(\mathrm{VI}(A,f,x))$ that takes into account a perturbation on all the data of the problem, not only on $x$ but also on $A$ and $f$. Precisely, we consider the following perturbed variational inequality which consists in finding $y(t) \in \H$ such that

\vspace{-0.19cm}
\begin{minipage}{.31\linewidth}
	\begin{equation*}
	\!\! \!\!\! \!\!\! (\mathrm{VI}( A(t,\cdot),f(t,\cdot),x(t) ))
	\end{equation*}
\end{minipage}
\begin{minipage}{.68\linewidth}
	\begin{equation*}
	\!\! \forall z \in \H, \;\; \langle A(t,y) , z - y \rangle + f(t,z) - f(t,y) \geq \langle x(t) , z - y \rangle ,
	\end{equation*}
\end{minipage}
\vspace{0.35cm}

where $A(t,\cdot)$, $f(t,\cdot)$ and $x(t)$ satisfy some appropriate assumptions that will be specified in Section~\ref{sec3}. Note that the solution of $(\mathrm{VI}( A(t,\cdot),f(t,\cdot),x(t) ))$ is given by
$$ y(t) = \prox_{A(t,\cdot),f(t,\cdot)} (x(t)). $$
Our main objective in this paper is to derive sufficient conditions on $A(t,\cdot)$, $f(t,\cdot)$ and $x(t)$ under which $y(t)$ is differentiable at $t=0$ and to provide an explicit formula for $y'(0)$. As mentioned in the previous section, R.T.~Rockafellar already dealt with the \textit{$t$-independent framework}, that is, with the particular case where $A(t,\cdot) = A$ and $f(t,\cdot) = f$ are $t$-independent. The concepts introduced by R.T.~Rockafellar cannot be directly applied to the \textit{$t$-dependent framework}. Therefore, we extend in Section~\ref{sec3} the notions of twice epi-differentiability, semi-differentiability and proto-differentiability to the case of parameterized functions, single-valued and set-valued maps respectively. These extensions are not a simple replica of the methodology developed by R.T.~Rockafellar and require several adjustments in the definitions. Moreover we show explicitly in Example~\ref{excontreex} that the situation is more complicated when dealing with the $t$-dependent setting. Precisely, the properness of the second epi-derivative of the one-variable function $f(x) = \vert x \vert$ is not preserved under the perturbation $f(t,x) = \vert x-t \vert$ (see Example~\ref{excontreex}), while it is preserved with the perturbation $f(t,x) = \vert x-t^2 \vert$ (see Example~\ref{excontreex2}). The difference between these two kinds of perturbations is the key point of the understanding of how to handle the $t$-dependent situation. We refer to Section~\ref{secnewassumption} for a detailed discussion about this issue. The new concept of \textit{convergent supporting hyperplane}, introduced in this paper (see Definition~\ref{defCSH}), turns out to be efficient in order to prove the equivalence between the twice epi-differentiability of $f(t,\cdot)$ and the proto-differentiability of its subdifferential $\partial f(t,\cdot)$ (see Theorem~\ref{thmmain1}). Actually the existence of a convergent supporting hyperplane to the second-order difference quotient is shown to be equivalent to the properness of the second epi-derivative of $f(t,\cdot)$ (see Proposition~\ref{propderniere}). Next we derive in Theorem~\ref{thmmain2} the proto-differentiability of $\prox_{A(t,\cdot),f(t,\cdot)}$ and we deduce in Theorem~\ref{thmmain3} that $y(t) = \prox_{A(t,\cdot),f(t,\cdot)} (x(t))$ is differentiable at $t=0$. Moreover we prove that $y'(0)$ can be expressed as the image of $x'(0)$ under a generalized proximity operator involving the semi-derivative of $A(t,\cdot)$ and the second epi-derivative of $f(t,\cdot)$. In particular, note that $y'(0)$ is then the unique solution of the associated variational inequality.

\medskip

Our main result (Theorem~\ref{thmmain3}) encompasses and extends Rockafellar's work, and thus the results of Mignot and Haraux, mentioned in the previous section. For instance, it permits to deal with the differentiability of $t \mapsto \proj_{\C(t)} (x(t))$ where the convex set $\C(t)$ is also perturbed in some sense (see Section~\ref{secappl3}). The originality of the present work is twofold. Firstly, it allows to deal with the sensitivity analysis of variational inequalities where all the data are $t$-dependent, via a nontrivial extension of the notion of twice epi-differentiability originally introduced by R.T. Rockafellar. Secondly, it permits to fill a gap in the literature, with respect to the works of A.J.~King, A.B.~Levy and R.T.~Rockafellar in \cite{KingRock,LevyRock}. In these references, the authors investigated the differentiability of the solutions set to parameterized generalized equations of the form
$$ 0 \in A(t,y)+B(t,y) $$
via the semi- and proto-derivatives of the involved single- and set-valued maps $A$ and $B$ respectively. However, no link has been explored with the twice epi-differentiability in the case where the set-valued map $B(t,\cdot) = \partial f(t,\cdot)$ coincides with the subdifferential of a parameterized convex function $f(t,\cdot)$.

\subsection{Applications and additional comments}
As an application of the theoretical results developed in this paper, we investigate the sensitivity analysis of parameterized convex optimization problems in Section~\ref{sec5}. More precisely, we first consider in Section~\ref{secappl1} the general parameterized minimization problem of the sum of two lower semi-continuous convex functions $f$ and $g$ ($g$ being smooth). We show, under some appropriate assumptions, that the derivative at $t=0$ of the perturbed solution is still a solution of a convex optimization problem involving the second epi-derivative of $f$ and a quadratic term involving the Hessian matrix of $g$ (see Proposition~\ref{propgenefin} for details). An illustrative example in one dimension is provided in Section~\ref{secappl2} showing both the nontriviality of the $t$-dependent setting and the applicability of our theoretical results. Finally, the sensitivity analysis of parameterized smooth convex optimization problems with inequality constraints in a finite-dimensional setting is discussed in details in Section~\ref{secappl3}.  

\medskip

We end the paper with two discussions in Section~\ref{sec6}. Firstly, we justify our choice of the formula of the second-order difference quotient of the twice epi-differentiability in the $t$-dependent setting (see Section~\ref{secform} for details). Secondly, we show that the equivalence between the twice epi-differentiabilities of a convex function and its conjugate is not preserved in the $t$-dependent case and requires some additional assumptions (that are automatically satisfied in the $t$-independent framework). We refer to Section~\ref{secdual} for details.

\subsection{Organization of the paper}
The paper is organized as follows. Section~\ref{secnot} is devoted to the main notations, definitions and auxiliary results from convex analysis. Section~\ref{sec3} is dedicated to the detailed setting of the paper and to the presentation of its main objective. Then, the extensions of the notions of generalized differentiability (semi-differentiability, proto-differentiability and twice epi-differentiability) to the $t$-dependent framework are specified. In Section~\ref{secmainresults}, we state and prove our main results (Theorems~\ref{thmmain1}, \ref{thmmain2} and \ref{thmmain3}) where the new notion of convergent supporting hyperplane is introduced. In Section~\ref{sec5}, as an application of our theoretical results, we investigate the sensitivity analysis of parameterized convex optimization problems. We conclude this paper with some additional comments in Section~\ref{sec6}.

\section{Recalls on convergence notions and convex analysis}\label{secnot}

In the next subsections we introduce some notations useful throughout the paper. We first give recalls about Painlev\'e-Kuratowski and Mosco convergences (see Section~\ref{secPKMconv}), graphical convergence (see Section~\ref{secGconv}) and epi-convergence (see Section~\ref{secEconv}). We continue with basics of convex analysis (see Section~\ref{secconvanal}) and we conclude with recalls about Attouch's theorems (see Section~\ref{secattouch}). We refer to standard books like~\cite{Attouch,JBHU,Penot,RockWets} and references therein.

\subsection{Two convergence modes of a parameterized family of subsets}\label{secPKMconv}
In this paper we denote by $\d (\cdot,S)$ the classical distance function to any subset $S$ of $\H$. 

\medskip

Let $(S_\t)_{\t > 0}$ be a parameterized family of subsets of $\H$. The \textit{outer}, \textit{weak-outer}, \textit{inner} and \textit{weak-inner limits} of $(S_\t)_{\t > 0}$ when $\t \to 0$ are respectively defined by
$$ \begin{array}{rcl}
\limsup S_\t & \!\! := \!\! & \{ x \in \H \mid \exists (t_n)_{n} \to 0, \; \exists (x_n)_{n} \rightarrow x, \; \forall n \in \N, \; x_n \in S_{t_n} \} , \\[5pt]
\wlimsup S_\t & \!\! := \!\! & \{ x \in \H \mid \exists (t_n)_{n} \to 0, \; \exists (x_n)_{n} \rightharpoonup x, \; \forall n \in \N, \; x_n \in S_{t_n} \} , \\[5pt]
\liminf S_\t & \!\! := \!\! & \{ x \in \H \mid \forall (t_n)_{n} \to 0, \; \exists (x_n)_{n} \rightarrow x, \; \exists N \in \N, \; \forall n \geq N, \; x_n \in S_{t_n}  \} , \\[5pt]
\wliminf S_\t & \!\! := \!\! & \{ x \in \H \mid \forall (t_n)_{n} \to 0, \; \exists (x_n)_{n} \rightharpoonup x, \; \exists N \in \N, \; \forall n \geq N, \; x_n \in S_{t_n}  \} ,
\end{array} $$
where $\to$ (resp. $\rightharpoonup$) denotes the strong (resp. weak) convergence in $\H$. Note that the four following inclusions always hold true:
$$ \liminf S_\t \subset \limsup S_\t \subset \wlimsup S_\t \;\; \text{and} \;\; \liminf S_\t \subset \wliminf S_\t \subset \wlimsup S_\t . $$
In the whole paper, note that all limits with respect to $\t$ will be considered for $\t \to 0$. For the ease of notations, when no confusion is possible, the notation $\t \to 0$ will be removed.

\begin{remark}\label{rkouterinnerlimit}
	Let $(S_\t)_{\t > 0}$ be a parameterized family of subsets of $\H$. It can be shown that $ \limsup S_\t = \{ x \in \H \mid \liminf \d (x,S_\t) = 0 \} $ and $ \liminf S_\t = \{ x \in \H \mid \limsup \d (x,S_\t) = 0 \} $ (see, e.g., \cite[Proposition~1.42 p.35]{Penot}). We can deduce that $\limsup S_\t$ and $\liminf S_\t$ are closed subsets of $\H$. 
\end{remark}

\begin{remark}\label{rkouterinnerlimit2}
	Convexity is a stable property under inner and weak-inner limits. Precisely, if $(S_\t)_{\t > 0}$ is a parameterized family of convex subsets of $\H$, then $\liminf S_\t$ and $\wliminf S_\t$ are also convex. This result is not true in general for $\limsup S_\t$ and $\wlimsup S_\t$ (see, e.g., \cite[p.119]{RockWets}).
\end{remark}

Let us recall the two following convergence modes for a parameterized family $(S_\t)_{\t > 0}$ of subsets of $\H$.

\begin{definition}[Painlev\'e-Kuratowski convergence]
	A parameterized family $(S_\t)_{\t > 0}$ of subsets of~$\H$ is said to be \textit{PK-convergent} if
	$$ \limsup S_\t \subset \liminf S_\t .$$
	In that case, we denote by $  \PKlim S_\t :=  \liminf S_\t =  \limsup S_\t $.
\end{definition}

\begin{definition}[Mosco convergence]
	A parameterized family $(S_\t)_{\t > 0}$ of subsets of~$\H$ is said to be \textit{M-convergent} if
	$$ \wlimsup S_\t \subset \liminf S_\t .$$
	In that case, we denote by $  \Mlim S_\t :=  \liminf S_\t = \limsup S_\t = \wliminf S_\t = \wlimsup S_\t $.
\end{definition}

\begin{remark}
	(i) Let $(S_\t)_{\t > 0}$ be a parameterized family of subsets of $\H$. If $(S_\t)_{\t > 0}$ M-converges, then $(S_\t)_{\t > 0}$ PK-converges. In that case $\PKlim S_\t = \Mlim S_\t  $. (ii) If $\H$ is finite-dimensional, the Painlev\'e-Kuratowski and the Mosco convergences clearly coincide.
\end{remark}

\subsection{Graphical convergence of a parameterized family of set-valued maps}\label{secGconv}
For a set-valued map $A : \H \rightrightarrows \H$ on $\H$, the {\it domain} of $A$ is given by $ \D (A) := \{ x \in \H \mid A(x) \neq \emptyset \}$ and its \textit{graph} is defined by $ \Gr (A) := \{ (x,y) \in \H \times \H \mid y \in A(x) \} $. We denote by $A^{-1} : \H \rightrightarrows \H$ the set-valued map defined by
$$ A^{-1} (y) := \{ x \in \H \mid y \in A(x) \}, $$
for all $y \in \H$. In particular, for all $x$, $y \in \H$, it holds that 
$$
(x,y) \in \Gr (A) \Longleftrightarrow y \in A(x) \Longleftrightarrow x \in A^{-1} (y) \Longleftrightarrow (y,x) \in \Gr (A^{-1}) .
$$

\begin{definition}[Graphical convergence]
	A parameterized family $(A_\t)_{\t > 0}$ of set-valued maps on $\H$ is said to be \textit{G-convergent} if $(\Gr (A_\t))_{\t > 0}$ is PK-convergent. In that case, we denote by $\glim A_\t : \H \rightrightarrows \H $ the set-valued map characterized by its graph as follows:
	$$ \Gr \left( \glim A_\t \right) := \PKlim \Gr ( A_\t ) . $$
\end{definition}

\begin{remark}
	Let $(A_\t)_{\t > 0}$ be a parameterized family of set-valued maps on $\H$. Then, $(A_\t)_{\t > 0}$ is G-convergent if and only if $(A^{-1}_\t)_{\t > 0}$ is G-convergent. In that case, it holds that $ \glim ( A^{-1}_\t ) = ( \glim A_\t )^{-1} $.
\end{remark}

\begin{remark}\label{propgraphconv}
	Let $(A_\t)_{\t > 0}$ be a $G$-convergent parameterized family of set-valued maps on~$\H$ and let $A := \glim A_\t$. If $(z_\t,\xi_\t)_{\t > 0} \to (z,\xi)$ with $(z_\t,\xi_\t) \in \Gr (A_\t)$ for all $\t > 0$, then $(z,\xi) \in \Gr (A)$.
\end{remark}

\subsection{M-convergence of a parameterized family of extended-real-valued functions}\label{secEconv}
Let $\overline{\R} := \R \cup \{ - \infty , +\infty \}$. For an extended-real-valued function $f : \H \to \overline{\R}$ on $\H$, the {\it domain} of $f$ is defined by $\dom (f) := \{ x \in \H \mid f(x) < +\infty \} $ and its \textit{epigraph} is given by
$$ \epi (f) := \{ (x,\lambda) \in \H \times \R \mid f(x) \leq \lambda \}. $$ 
We denote by $\E_\H$ the \textit{set of all epigraphs} on $\H$, that is, $S \in \E_\H$ if and only if there exists an extended-real-valued function $f : \H \to \overline{\R}$ such that $S = \epi (f)$. 


\begin{remark}
	The set $\E_\H$ is stable under outer and inner limits (see, e.g., \cite[p.240]{RockWets}). Precisely, if $(S_\t)_{\t > 0}$ is a parameterized family of subsets of $\H \times \R$ such that $S_\t \in \E_\H$ for all $\t > 0$, then $\lim\sup S_\t$ and $\lim\inf S_\t$ both belong to $\E_\H$.
\end{remark}

\begin{definition}[M-convergence]
	A parameterized family $(f_\t)_{\t > 0}$ of extended-real -valued functions on $\H$ is said to be M-convergent if $(\epi (f_\t))_{\t > 0}$ is M-convergent. In that case, we denote by $\elim f_\t : \H \to \overline{\R} $ the extended-real-valued function characterized by its epigraph as follows:
	$$ \epi ( \elim f_\t ) := \Mlim \epi ( f_\t ) . $$
\end{definition}

Recall the following characterization of epi-convergence. We refer to \cite[Proposition~3.19 p.297]{Attouch} or \cite[Proposition~7.2 p.241]{RockWets} for details.

\begin{proposition}\label{propcharact}
	Let $f$ be an extended-real-valued function on $\H$ and let $(f_\t)_{\t > 0}$ be a parameterized family of extended-real-valued functions on $\H$. Then $(f_\t)_{\t > 0}$ M-converges with $f = \elim f_\t $ if and only if, for all $z \in \H$, there exists $(z_\t )_{\t > 0} \to z$ such that $\limsup f_\t (z_\t) \leq f(z)$ and, for all $(z_\t )_{\t > 0} \rightharpoonup z$, $\liminf f_\t (z_\t) \geq f(z)$.
\end{proposition}

\subsection{Basics of convex analysis}\label{secconvanal}

A set-valued operator $A : \H \rightrightarrows \H$ is said to be \textit{monotone} if
$$ \forall (x_1,y_1), (x_2,y_2) \in \Gr (A), \quad \langle y_2 - y_1 , x_2 - x_1 \rangle \geq 0 . $$
Moreover, $A$ is said to be \textit{maximal monotone} if $\Gr(A) \subset \Gr (B)$ for some monotone set-valued operator $B : \H \rightrightarrows \H$ implies that $A=B$.

\medskip

A set-valued map $A : \H \rightrightarrows \H$ is said to be \textit{single-valued} if $A(x)$ is a singleton for all $x \in \H$. In that case, it holds in particular that $\D (A) = \H$ and we denote by $A : \H \rightarrow \H$ (instead of $A : \H \rightrightarrows \H$). For a single-valued map $A : \H \rightarrow \H$, we say that $A$ is \textit{Lipschitz continuous} if
$$ \exists M \geq 0 , \quad \forall x_1,x_2 \in \H, \quad \Vert A(x_2) - A(x_1) \Vert \leq M \Vert x_2 - x_1 \Vert , $$
and we say that $A$ is \textit{strongly monotone} if
$$ \exists \alpha > 0 , \quad \forall x_1,x_2 \in \H, \quad  \langle A(x_2) - A(x_1) , x_2 - x_1 \rangle \geq \alpha \Vert x_2 - x_1 \Vert^2 . $$
In the whole paper, we denote by $\mathcal{A}(\H)$ the set of all single-valued maps $A : \H \rightarrow \H$ that are Lipschitz continuous and strongly monotone.

\medskip

For an extended-real-valued function $f : \H \to \R \cup \{ +\infty \}$, we say that $f$ is \textit{proper} if $\dom (f) \neq \emptyset$, and that $f$ is \textit{lower semi-continuous} if its epigraph $\epi (f)$ is closed in $\H \times \R$. Finally, recall that $f$ is a convex function if and only if its epigraph $\epi (f)$ is a convex subset of $\H \times \R$. In the whole paper, we denote by $\Gamma_0 (\H)$ the set of all extended-real-valued functions $f : \H \to \R \cup \{ +\infty \}$ that are proper, lower semi-continuous and convex.

\medskip

Let $f \in \Gamma_0 (\H)$. We denote by $\partial f : \H \rightrightarrows \H$ the {\it subdifferential operator} of $f$ defined by
$$ \partial f (x) := \{ y \in \H \mid \forall z \in \H, \; \langle y , z - x \rangle \leq f(z) - f(x) \}, $$
for all $x \in \H$. In particular, note that $x \in \argmin f$ if and only if $0 \in \partial f (x)$. Recall that $\partial f$ is a maximal monotone operator (see \cite{RockPacific} or \cite[Theorem~12.17 p.542]{RockWets} for details). Moreover, it follows from the classical Br\o ndsted-Rockafellar theorem (see, e.g., \cite[Theorem~6.5 p.333]{Penot2}) that $\D (\partial f) \neq \emptyset$, and thus $f$ admits a \textit{supporting hyperplane}, that is,
$$ \exists (z,\xi,\beta) \in \H \times \H \times \R, \quad \left\lbrace 
\begin{array}{l}
\forall w \in \H, \quad f(w) \geq \langle \xi , w \rangle + \beta , \\[5pt]
f(z) = \langle \xi , z \rangle + \beta.
\end{array}
\right. $$
Note that $(z,\xi,\beta) \in \H \times \H \times \R$ is a supporting hyperplane of $f$ if and only if $\xi \in \partial f (z)$ and $\beta = f(z)-\langle \xi , z \rangle$. 

\medskip

The {\it generalized proximity} (also well-known as \textit{proximal}) {\it operator}  associated to some $(A,f) \in \mathcal{A}(\H) \times \Gamma_0 (\H)$ is defined by
$$ 
\prox_{A,f} := (A +\partial f )^{-1}. 
$$
For all $x$, $y \in \H$, one can easily see that $y \in \prox_{A,f} (x)$ if and only if $y$ is a solution of the \textit{nonlinear variational inequality of second kind} given by

\vspace{-0.19cm}\noindent
\begin{minipage}{.16\linewidth}
	\begin{equation*}
	\!\!\! (\mathrm{VI}(A,f,x))
	\end{equation*}
\end{minipage}
\begin{minipage}{.8\linewidth}
	\begin{equation*}
	\forall z \in \H, \quad \langle A(y) , z - y \rangle + f(z) - f(y) \geq \langle x , z - y \rangle.
	\end{equation*}
\end{minipage}\vspace{0.35cm}

From the contraction mapping principle, it can be proved that the above variational inequality~$(\mathrm{VI}(A,f,x))$ admits a unique solution (see, e.g., \cite[Proposition~31 p.140]{Brezisarticle} for details) given by $\prox_{A,f}(x)$. We deduce that $\prox_{A,f} : \H \to \H$ is a single-valued map. Moreover, it can be easily proved that $\prox_{A,f} $ is $\frac{1}{\alpha}$-Lipschitz continuous where $\alpha$ denotes the strong monotonicity coefficient of $A$.

%
%

\subsection{Recalls about Attouch's theorems}\label{secattouch}
We conclude this section with two crucial results stated by H. Attouch in \cite[Theorem~3.66 p.373 and Corollary~3.65 p.372]{Attouch}. The first result is well-known as the classical Attouch's theorem.

\begin{theorem}\label{thmattouch}
	Let $(f_\t)_{\t > 0}$ be a parameterized family of functions in $\Gamma_0(\H)$ and let $f \in \Gamma_0(\H)$. Then, $f = \elim f_\t$ if and only if the two following assertions are satisfied:
	\begin{enumerate}
		\item[{\rm (i)}] $\partial f = \glim \partial f_\t$;
		\item[{\rm (ii)}] there exists $(z_\t, \xi_\t)_{\t > 0} \to (z,\xi)$ such that $(z_\t , \xi_\t) \in \Gr ( \partial f_\t)$ for all $\t > 0$ and $f_\t(z_\t) \to  f(z)$.
	\end{enumerate}
\end{theorem}

\begin{theorem}\label{thmattouch2}
	Let $(f_\t)_{\t > 0}$ be a parameterized family of functions in $\Gamma_0(\H)$. If $(\partial f_\t)_{\t > 0}$ is G-convergent and if $\glim \partial f_\t$ is a maximal monotone operator, then there exists $f \in \Gamma_0(\H)$ such that $\glim \partial f_\t = \partial f$.
\end{theorem}

\begin{remark}\label{remfinite}
	Let us consider the framework of Theorem~\ref{thmattouch}. If the equivalent assertions are satisfied, note that $(z,\xi) \in \Gr (\partial f)$ from Remark~\ref{propgraphconv}. In particular note that $f(z) \in \R$.
\end{remark}

\section{Setting, objective and generalized notions of differentiability}\label{sec3}
Before exhibiting in details the setting of the present paper and its main objective, we need to introduce first some notations. 

\medskip

Let $A : \R_+ \times \H \rightrightarrows \H$ be a parameterized set-valued map. For the simplicity of notations, we denote by $A^{-1} : \R_+ \times \H \rightrightarrows \H $ the parameterized set-valued map defined by 
$$A^{-1} (t,y) := (A(t,\cdot))^{-1} (y)$$ 
for all $(t,y) \in \R_+ \times \H$.

\medskip

A parameterized set-valued map $A : \R_+ \times \H \rightrightarrows \H$ is said to be a \textit{parameterized single-valued map} if $A(t,\cdot) : \H \to \H$ is a single-valued map for all $t \geq 0$. In that case we denote by $A : \R_+ \times \H \to \H$ (instead of $A : \R_+ \times \H \rightrightarrows \H$). 

\medskip

Let us introduce:
\begin{itemize}
	\item $\mathcal{A}(\cdot,\H)$ the set of all parameterized single-valued maps $A : \R_+ \times \H \to \H$ such that $A(t,\cdot) \in \mathcal{A}(\H)$ for all $t \geq 0$;
	\item $\Gamma_0(\cdot,\H)$ the set of all parameterized extended-real-valued functions $f : \R_+ \times \H \rightarrow \R \cup \{ +\infty \}$ such that $f(t,\cdot) \in \Gamma_0(\H)$ for all $t \geq 0$.
\end{itemize}
Let $f \in \Gamma_0(\cdot,\H)$. For the ease of notations, we denote by 
$$  \partial f (t,x) := \partial ( f(t,\cdot) ) (x) $$ 
for all $(t,x) \in \R_+ \times \H$. Then $\partial f : \R_+ \times \H \rightrightarrows \H$ is a parameterized set-valued map. Moreover, we introduce the notation
$$ f^{-1} (\cdot , \R) := \{ x \in \H \mid \forall t \geq 0, \; f(t,x) \in \R \} . $$
If $f$ is $t$-independent, note that $f \in \Gamma_0(\H)$ and that $f^{-1}(\cdot,\R)$ coincides with the classical notation $f^{-1}(\R)$.

\subsection{Setting and main objective of the paper}\label{secsetting}

In this paper we focus on the sensitivity analysis, with respect to the parameter $t \geq 0$, of the general nonlinear variational inequality of second kind given by

\vspace{-0.19cm}
\begin{minipage}{.31\linewidth}
	\begin{equation*}
	\!\! \!\!\! \!\!\! ( \mathrm{VI} ( A(t,\cdot),f(t,\cdot),x(t) ) )
	\end{equation*}
\end{minipage}
\begin{minipage}{.68\linewidth}
	\begin{equation*}
	\!\! \forall z \in \H, \;\; \langle A(t,y) , z - y \rangle + f(t,z) - f(t,y) \geq \langle x(t) , z - y \rangle ,
	\end{equation*}
\end{minipage}
\vspace{0.35cm}

%
where $(A,f) \in \mathcal{A}(\cdot,\H) \times \Gamma_0(\cdot,\H)$ and where $x : \R_+ \to \H$ is a given function. From Section~\ref{secconvanal}, the above variational inequality admits for all $t \geq 0$ a unique solution $y(t) \in \H$ given by
$$ y(t) = \prox_{A(t,\cdot),f(t,\cdot)} (x(t)). $$
Our main objective in this paper is to derive sufficient conditions on $A$, $f$ and $x$ under which $y : \R_+ \to \H$ is differentiable at $t=0$ and to provide an explicit formula for $y'(0)$. 

\medskip

In the literature numerous generalized notions of differentiability have been introduced in order to deal with nonsmooth functions. In particular, notions of \textit{semi-differentiability}, \textit{twice epi-differentiability} and \textit{proto-differentiability} have been introduced respectively in \cite{Penotarticle}, \cite{Rockepidiff} and~\cite{Rockproto}. These wonderful tools turned out to be sufficient in order to deal with the case where $A$ and $f$ are $t$-independent. We refer to \cite{Rocksecond,RockWets} for the finite-dimensional case and to \cite{Chi} for the infinite-dimensional case.\footnote{These three references are interested only in the $t$-independent framework and actually they also only consider the case where $A = \I$ is the identity operator. Precisely, they proved that if $f$ is twice epi-differentiable, then the classical Moreau's proximity operator $\prox_{f} = \prox_{\I,f}$ is proto-differentiable. Even if it is not mentioned in \cite{Chi,Rocksecond,RockWets}, one can easily deduce from the proto-differentiability of $\prox_{f}$ that $y$ is differentiable at $t=0$ and provide an explicit formula for $y'(0)$. The results presented in the present paper encompass and extend these results.}

\medskip

In view of dealing with the $t$-dependence of $A$ and $f$, we need to extend the notions of semi-differentiability, twice epi-differentiability and proto-differentiability above mentioned to the $t$-dependent framework. This is the aim of Sections~\ref{secsemidiff}, \ref{secprotodiff} and \ref{sectwiceepidiff}. 

\begin{remark}\normalfont
	We mention here that these generalizations are not trivial and are not sufficient in order to fully adapt the strategy developed in \cite{Chi,Rocksecond,RockWets} to the $t$-dependent framework. Indeed, we will provide a simple example (see Example~\ref{excontreex}) showing that the situation is more complicated in the $t$-dependent setting. As a consequence, we will introduce in Section~\ref{secnewassumption} a new concept called \textit{convergent supporting hyperplane} (see Definition~\ref{defCSH}) that allows us to conclude. Actually, this condition turns out to be necessary and sufficient in a sense that can be made precise (see Proposition~\ref{propderniere}). We refer to Section~\ref{secnewassumption} for a detailed discussion about this issue. Similarly, we will also prove that, in contrary to the $t$-independent framework, the generalized twice epi-differentiability of a function $f \in \Gamma_0(\cdot,\H)$ cannot be directly related to the twice epi-differentiability of its conjugate function~$f^*$. We refer to Section~\ref{secdual} for more details.
\end{remark}

\subsection{Semi-differentiability of a parameterized single-valued map}\label{secsemidiff}

In this section our aim is to generalize the classical notion of semi-differentiability to the case where the single-valued map considered depends on the parameter $t \geq 0$. 

\begin{definition}[Semi-differentiability]\label{defsemidiff}
	Let $A : \R_+ \times \H \to \H$ be a parameterized single-valued map and let $x \in \H$. If the limit
	$$ D_s A(x)(w) := \lim\limits_{\substack{\t \to 0 \\ w' \to w}} \dfrac{A(\t,x+\t w')- A(0,x)}{\t} $$
	exists in $\H$ for all $w \in \H$, we say that $A$ is \textit{semi-differentiable} at $x$. In that case, $D_s A(x) : \H \to \H$ is a single-valued map called the \textit{semi-derivative} of $A$ at $x$.
\end{definition}

If the single-valued map $A$ is $t$-independent, Definition~\ref{defsemidiff} recovers the classical notion of semi-differentiability originally introduced in~\cite{Penotarticle}.

\medskip

In the sequel we will denote by $\mathcal{A}_{\mathrm{unif}}(\cdot,\H)$ the set of all parameterized single-valued maps $A : \R_+ \times \H \to \H$ such that $A$ is \textit{uniformly Lipschitz continuous}, that is,
$$ \exists M \geq 0 , \quad \forall t \geq 0, \quad \forall x_1,x_2 \in \H, \quad \Vert A(t,x_2) - A(t,x_1) \Vert \leq M \Vert x_2 - x_1 \Vert , $$
and \textit{uniformly strongly monotone}, that is,
$$ \exists \alpha > 0 , \quad \forall t \geq 0, \quad \forall x_1,x_2 \in \H, \quad  \langle A(t,x_2) - A(t,x_1) , x_2 - x_1 \rangle \geq \alpha \Vert x_2 - x_1 \Vert^2 .  $$

\begin{proposition}\label{propDA}
	Let $A \in \mathcal{A}_{\mathrm{unif}}(\cdot,\H)$. If $A$ is semi-differentiable at $x \in \H$, then $D_s A(x) \in \mathcal{A}(\H)$.
\end{proposition}

\begin{proof}
	Let us prove that $D_s A(x)$ is Lipschitz continuous. Let $w_1$, $w_2 \in \H$. It holds that $ \Vert A(\t,x+\t w_2) - A(\t,x+\t w_1) \Vert \leq M \t \Vert w_2 - w_1 \Vert$ and thus
	$$ \left\Vert \dfrac{A(\t,x+\t w_2) - A(0,x)}{\t} - \dfrac{A(\t,x+\t w_1) - A(0,x)}{\t} \right\Vert \leq M \Vert w_2 - w_1 \Vert , $$ 
	for all $\t > 0$. Letting $\t \to 0$ leads to $ \Vert D_s A(x)(w_2)-D_sA(x)(w_1) \Vert \leq M \Vert w_2 - w_1 \Vert$. We conclude that $D_s A(x)$ is Lipschitz continuous. The proof is similar for proving the strong monotonicity of $D_s A(x)$.
\end{proof}

\begin{remark}\label{remunifconstant}\normalfont
	Let $A \in \mathcal{A}_{\mathrm{unif}}(\cdot,\H)$ be semi-differentiable at $x \in \H$. If we denote by $M \geq 0$ and $\alpha > 0$ the uniform Lipschitz and strong monotonicity coefficients of $A$, then $D_s A(x) \in \mathcal{A}(\H)$ is $M$-Lipschitz continuous and $\alpha$-strongly monotone.
\end{remark}

\subsection{Proto-differentiability of a parameterized set-valued map}\label{secprotodiff}

In this section our aim is to generalize the classical notion of proto-differentiability to the case where the set-valued map considered depends on the parameter $t \geq 0$. 

\medskip

Let $A : \R_+ \times \H \rightrightarrows \H$ be a parameterized set-valued map. For all $\t > 0$, $x \in \H$ and $v \in A(0,x)$, we denote by
$$ \fonctionset{\Delta_\t A (x | v)}{\H}{\H}{w}{ \Delta_\t A (x | v)(w) := \dfrac{A (\t,x+\t w) - v}{\t}. }  $$

\begin{definition}[Proto-differentiability]\label{defprotodiff}
	Let $A : \R_+ \times \H \rightrightarrows \H$ be a parameterized set-valued map. We say that $A$ is \textit{proto-differentiable} at $x \in \H$ for $v \in A(0,x)$ if $(\Delta_\t A (x | v))_{\t > 0}$ G-converges. In that case, we denote by
	$$ D_p A(x | v) := \glim \Delta_\t A (x | v) $$
	the set-valued map $D_p A (x | v) : \H \rightrightarrows \H$ called the \textit{proto-derivative} of $A$ at $x$ for $v$.
\end{definition}

If the set-valued map $A$ is $t$-independent, Definition~\ref{defprotodiff} recovers the classical notion of proto-differentiability originally introduced in~\cite{Rockproto}.

\begin{remark}
	Let us point out that if a parameterized single-valued map $A : \R_+ \times \H \to \H$ is semi-differentiable at $x \in \H$, then $A$ is proto-differentiable at $x$ for $A(0,x)$ with $ D_p A(x | A(0,x) ) = D_s A(x) $.
\end{remark}

One can easily prove the two following results. In the $t$-independent case, we recover~\cite[p.331-333]{RockWets}.

\begin{proposition}\label{prop1}
	Let $A : \R_+ \times \H \to \H$ be a parameterized single-valued map and $B : \R_+ \times \H \rightrightarrows \H$ be a parameterized set-valued map. Let $x \in \H$ and $v \in A(0,x) + B(0,x)$. If $A$ is semi-differentiable at $x$, then $A+B$ is proto-differentiable at $x$ for $v$ if and only if $B$ is proto-differentiable at $x$ for $v-A(0,x)$. In that case it holds that
	$$ D_p (A+B)(x | v) = D_s A (x) + D_p B(x|v-A(0,x)). $$
\end{proposition}

\begin{proposition}\label{prop2}
	Let $A : \R_+ \times \H \rightrightarrows \H$ be a parameterized set-valued map, $x \in \H$ and $v \in A(0,x)$. Then, $A$ is proto-differentiable at $x $ for $v $ if and only if $A^{-1}$ is proto-differentiable at $v$ for $x$. In that case, it holds that
	$$ D_p ( A^{-1} ) (v | x) := ( D_p  A (x|v) )^{-1}. $$
\end{proposition}

\subsection{Twice epi-differentiability of parameterized extended-real-valued functions}\label{sectwiceepidiff}
In this section our aim is to generalize the classical notion of twice epi-differentiability to the case where the extended-real-valued function considered depends on the parameter $t \geq 0$. 

\medskip

Let $f  \in \Gamma_0(\cdot,\H)$. For all $\t > 0$, $x \in f^{-1} (\cdot , \R)$ and $v \in \partial f(0,x)$, we denote by
\begin{equation}\label{form}
\fonction{\Delta^2_\t f (x | v)}{\H}{\R \cup \{ +\infty \} }{w}{ \Delta^2_\t f (x | v)(w) := \dfrac{f (\t,x+\t w) - f(\t,x) - \t \langle v , w \rangle}{ \t^2 }. }
\end{equation}

\begin{definition}[Twice epi-differentiability]\label{deftwiceepidiff}
	Let $f  \in \Gamma_0(\cdot,\H)$. We say that $f$ is \textit{twice epi-differentiable} at $x \in f^{-1} (\cdot , \R)$ for $v \in \partial f(0,x)$ if $(\Delta^2_\t f (x | v))_{\t > 0}$ M-converges. In that case, we denote by
	$$ d^2_e f (x | v) := \elim \Delta^2_\t f (x | v) $$
	the extended-real-valued function $d^2_e f (x | v) : \H \to \overline{\R}$ called the \textit{second epi-derivative} of~$f$ at $x$ for~$v$.
\end{definition}

If the extended-real-valued function $f$ is $t$-independent, Definition~\ref{deftwiceepidiff} recovers the classical notion of twice epi-differentiability originally introduced in~\cite{Rockepidiff} (up to the multiplicative constant $\frac{1}{2}$).

\begin{remark}\normalfont
	Let $f  \in \Gamma_0(\cdot,\H)$ be twice epi-differentiable at $x \in f^{-1} (\cdot , \R)$ for $v \in \partial f(0,x)$. Even if $\Delta^2_\t f(x|v)$ is with values in $\R \cup \{ +\infty \}$ for all $\t > 0$, it may be possible that there exists $w \in \H$ such that $d^2_e f (x | v)(w) = -\infty$. We refer to Example~\ref{excontreex}. This is an important difference with the $t$-independent framework. We refer to Section~\ref{secnewassumption} for a detailed discussion about this issue.
\end{remark}

\begin{remark}\label{rem65}\normalfont
	Let $f  \in \Gamma_0(\cdot,\H)$ be twice epi-differentiable at $x \in f^{-1} (\cdot , \R)$ for $v \in \partial f(0,x)$. It might be possible that $d^2_e f (x | v)$ is not positively homogeneous of degree two. We refer to Example~\ref{excontreex2}. Indeed, note that $\Delta^2_{ \t} f (x | v) (\lambda w) \neq \lambda^2 \Delta^2_{\lambda \t} f (x | v) ( w)$ in general for $\lambda > 0$ and $w \in \H$. This is a second important difference with the $t$-independent framework.
\end{remark}

\begin{remark}\label{rempointwise}\normalfont
	Let $f : \R_+ \times \H \rightarrow \R$ be a function of class $C^2$ on $\R_+ \times \H$. A Taylor expansion of $f$ around $x \in \H$ with $v = \nabla_x f (0,x)$ implies that $(\Delta^2_\t f (x | v) (w))_{\t > 0}$ converges pointwise to
	$$ \dfrac{1}{2} \left\langle \nabla^2_{xx} f (0,x)(w), w \right\rangle + \left\langle \nabla^2_{tx} f(0,x), w \right\rangle  $$
	for all $w \in \H$. This remark is coherent with Remark~\ref{rem65} since the above expression is not positively homogeneous of degree two with respect to the variable $w$. 
\end{remark}

\begin{remark}\normalfont
	Note that Section~\ref{secform} will be devoted to a discussion about the pointwise convergence mentioned in Remark~\ref{rempointwise} and about the choice of  Formula~\eqref{form}.
\end{remark}

\section{Main results}\label{secmainresults}

The whole paper is based on Attouch's theorems (see Theorems~\ref{thmattouch} and \ref{thmattouch2}) and on Proposition~\ref{prop12} below. Its proof is simple and similar to the $t$-independent case (see, e.g., \cite[Proposition~2.7]{Rock122}).

\begin{proposition}\label{prop12}
	Let $f  \in \Gamma_0(\cdot,\H)$, $x \in f^{-1} (\cdot , \R)$ and $v \in \partial f(0,x)$. Then, $\Delta^2_\t f (x | v) $ belongs to $\Gamma_0(\H)$ with
	$$ \partial \Big( \Delta^2_\t f (x | v) \Big) = \Delta_\t (\partial f) (x | v), $$
	for all $\t > 0$.
\end{proposition}

One can easily deduce that the twice epi-differentiability of a function $f \in \Gamma_0(\cdot,\H)$ is strongly related to the proto-differentiability of its subdifferential operator $\partial f$. 

\medskip

This relation in the $t$-independent framework is simple and recalled in Section~\ref{secnewassumption} (see Proposition~\ref{proptinde}). However, a simple example (see Example~\ref{excontreex}) shows that the situation is more complicated in the $t$-dependent setting. This is the reason why we introduce a new concept called \textit{convergent supporting hyperplane} (see Definition~\ref{defCSH}) that allows us to state and prove the counterpart of Proposition~\ref{proptinde} in the $t$-dependent framework (see Theorem~\ref{thmmain1}). This condition actually turns out to be necessary and sufficient in a sense that can be made precise (see Proposition~\ref{propderniere}).

\medskip

Finally Section~\ref{secmain2} is devoted to the proto-differentiability of the generalized proximity operator (see Theorem~\ref{thmmain2}) and to our initial motivation, that is, to the sensitivity analysis of general nonlinear variational inequalities of second kind (see Theorem~\ref{thmmain3}).

\subsection{Convergent supporting hyperplane}\label{secnewassumption}

In the $t$-independent framework, the following proposition relating the twice epi-differentiability of a function $f \in \Gamma_0(\H)$ and the proto-differentiability of its subdifferential operator $\partial f$ is a well-known result. We refer to \cite{Rocksecond,RockWets} for the finite-dimensional case and to \cite{Chi} for the infinite-dimensional one. 

\begin{proposition}\label{proptinde}
	Let $f \in  \Gamma_0(\H)$ (that is $t$-independent), $x \in f^{-1}(\R)$ and $v \in \partial f(x)$. The following assertions are equivalent:
	\begin{enumerate}
		\item[{\rm (i)}] $f$ is twice epi-differentiable at $x$ for $v$;
		\item[{\rm (ii)}] $\partial f$ is proto-differentiable at $x$ for $v$ and $D_p (\partial f) (x | v)$ is a maximal monotone operator.
	\end{enumerate}
	In that case $d^2_e f(x|v) $ belongs to $\Gamma_0(\H)$ with $d^2_e f (x | v)(0) = 0$ and 
	$$ D_p (\partial f)(x|v) = \partial ( d^2_e f(x|v) ).$$
\end{proposition}

This classical result is based on Attouch's theorems (see Theorems~\ref{thmattouch} and \ref{thmattouch2}) and on Proposition~\ref{prop12}. However it is worth to note here that Proposition~\ref{proptinde} is also based on the following lemma. 

\begin{lemma}\label{lemtinde}
	Let $f \in  \Gamma_0(\H)$ (that is $t$-independent), $x \in f^{-1}(\R)$ and $v \in \partial f(x)$. Then, $\Delta^2_\t f(x|v)(0) = 0$ and $\Delta^2_\t f(x|v)(w) \geq 0$ for all $w \in \H$ and all $\t > 0$. In particular, it holds that $0 \in \partial ( \Delta^2_\t f (x|v) ) (0)$ for all $\t > 0$.
\end{lemma}

Let us recall the proof of Proposition~\ref{proptinde} in order to highlight the crucial role of Lemma~\ref{lemtinde}.

\begin{proof}[Proof of Proposition~\ref{proptinde}]
	Let us assume that $f$ is twice epi-differentiable at $x$ for $v$. Our first aim is to prove that $d^2_e f(x|v) \in \Gamma_0(\H)$. Firstly, let us assume by contradiction that there exists $w \in \H$ such that $d^2_e f (x | v)(w)=-\infty$. Since $(w,-1) \in \epi (  d^2_e f (x | v) ) = \limsup \epi ( \Delta^2_\t f (x | v) )$, there exist $(t_n) \to 0$ and $(w_n,\lambda_n) \to (w,-1)$ such that $(w_n,\lambda_n) \in \epi ( \Delta^2_{t_n} f (x | v) )$, that is, $ \Delta^2_{t_n} f (x | v) (w_n) \leq \lambda_n$ for all $n \in \N$. For sufficiently large $n$, it holds that $ \Delta^2_{t_n} f (x | v) (w_n) \leq \lambda_n < 0$ which raises a contradiction with Lemma~\ref{lemtinde}. Secondly, since $\epi (  d^2_e f (x | v) ) = \liminf \epi ( \Delta^2_\t f (x | v) )$, we easily conclude that $d^2_e f (x | v)$ is lower semi-continuous and convex (see Remarks~\ref{rkouterinnerlimit} and \ref{rkouterinnerlimit2} and Proposition~\ref{prop12}). Thirdly, since $(0,0) \in \epi ( \Delta^2_\t f (x | v) )$ for all $\t > 0$, we deduce that $(0,0) \in \epi ( d^2_e f (x | v) )$ and thus $d^2_e f (x | v) (0) \leq 0$. This concludes that $d^2_e f(x|v)$ is proper. Our second aim is to prove that $d^2_e f (x | v)(0) = 0$. By contradiction, let us assume that $d^2_e f (x | v)(0) < 0$. Then, there exists $\varepsilon > 0$ such that $(0,-\varepsilon) \in \epi ( d^2_e f (x | v) ) = \limsup \epi ( \Delta^2_\t f (x | v) )$. Thus, there exist $(t_n) \to 0$ and $(w_n,\varepsilon_n) \to (0,\varepsilon)$ such that $(w_n,-\varepsilon_n) \in \epi ( \Delta^2_{t_n} f (x | v) )$, that is, $ \Delta^2_{t_n} f (x | v) (w_n) \leq - \varepsilon_n$ for all $n \in \N$. For sufficiently large $n$, it holds that $ \Delta^2_{t_n} f (x | v) (w_n) \leq - \varepsilon_n < 0$ which constitutes a contradiction with Lemma~\ref{lemtinde}. Our last aim is to prove that $\partial f$ is proto-differentiable at $x$ for $v$ with $ D_p (\partial f)(x|v) = \partial ( d^2_e f(x|v) )$. This result directly follows from Proposition~\ref{prop12} and Theorem~\ref{thmattouch}.
	
	\medskip
	
	Now let us assume that $\partial f$ is proto-differentiable at $x$ for $v$ and $D_p (\partial f) (x | v)$ is a maximal monotone operator. From Proposition~\ref{prop12}, we get that $( \partial ( \Delta^2_\t f (x | v) ) )_{\t > 0} = ( \Delta_\t (\partial f) (x | v) )_{\t > 0}$ G-converges to the maximal monotone operator $ D_p (\partial f)(x|v)$. We deduce from Theorem~\ref{thmattouch2} that $ D_p (\partial f)(x|v) = \partial \varphi$ for some $\varphi \in \Gamma_0(\H)$. Our aim is now to apply Theorem~\ref{thmattouch} in order to conclude that $( \Delta^2_\t f (x | v) ) _{\t > 0}$ M-converges. Let us consider $z_\t := 0$ and $\xi_\t := 0$ for all $\t > 0$. In particular, we have $(z_\t,\xi_\t) \to (0,0)$ and since $0 \in \partial ( \Delta^2_\t f (x | v) ) (0)$ (see Lemma~\ref{lemtinde}), we deduce that $(z_\t,\xi_\t) \in \Gr ( \partial ( \Delta^2_\t f (x | v) ) )$ for all $\t > 0$. From Remark~\ref{propgraphconv}, we deduce that $(0,0) \in \Gr (\partial \varphi)$ and thus $0 \in \dom (\varphi)$. Considering $\psi := \varphi - \varphi (0) \in \Gamma_0(\H)$, we have $\psi (0) = 0$ and $\partial \psi = \partial \varphi$. Since $\Delta^2_\t f(x|v)(\xi_\t) = 0 \to \psi (0)$, the proof is complete from Theorem~\ref{thmattouch}.
\end{proof}

However it should be noted that Lemma~\ref{lemtinde} is not true in the general $t$-dependent setting. As a consequence the above proof cannot be fully adapted to the $t$-dependent framework. Actually, in the $t$-dependent setting, we can even prove that the above proposition is not true with a simple counter-example provided below. 

%

\begin{example}\label{excontreex}\normalfont
	Let $\H = \R$ and $f(t,x) := \vert x - t \vert$ for all $(t,x) \in \R_+ \times \R$. Let us consider $x =0$ and $v = 0 \in \partial f (0,x)$. One can easily compute that
	$$ \Delta^2_\t f(x|v) (w) = \dfrac{ \vert w - 1 \vert - 1 }{\t}, $$
	for all $\t > 0$ and all $w \in \R$. One can easily deduce that $f$ is twice epi-differentiable at $x$ for $v$ with 
	$$ d^2_e f(x|v)(w) = \left\lbrace 
	\begin{array}{lcr}
	-\infty & \text{if} & w \in [0,2], \\
	+\infty & \text{if} & w \notin [0,2].
	\end{array}
	\right. $$
	In particular note that $d^2_e f(x|v) \notin \Gamma_0(\H)$ and that $d^2_e f(x|v)(0) \neq 0$.
\end{example} 

We deduce from Example~\ref{excontreex} that Proposition~\ref{proptinde} does not admit an exact counterpart in the $t$-dependent framework. Thus our aim is now to introduce a new concept that allows to recover Proposition~\ref{proptinde} in the $t$-dependent framework. We refer to Section~\ref{secconvanal} for the notion of \textit{supporting hyperplane}.

\begin{definition}[Convergent supporting hyperplane]\label{defCSH}
	Let $(f_\t)_{\t > 0}$ be a parameterized family of functions in $\Gamma_0(\H)$. We say that $(f_\t)_{\t > 0}$ admits a \textit{convergent supporting hyperplane} if there exists $(z_\t,\xi_\t,\beta_\t)_{\t > 0} \to (z,\xi,\beta)$ such that $(z_\t,\xi_\t,\beta_\t)$ is a supporting hyperplane of $f_\t$ for all $\t > 0$.
\end{definition}

\begin{remark}\label{remCSH}
	Let us point out that the existence of a convergent supporting hyperplane to a parameterized family $(f_\t)_{\t > 0}$ of functions in $\Gamma_0(\H)$ is equivalent to the existence of $(z_\t,\xi_\t)_{\t > 0} \to (z,\xi)$ such that $(z_\tau,\xi_\tau) \in \mathrm{Gr}(\partial f_\tau)$ for all $\tau > 0$ and such that $f_\tau(z_\tau) \to \gamma$ for some $\gamma \in \R$.
\end{remark}

We are now in position in order to state and prove the counterpart of Proposition~\ref{proptinde} in the $t$-dependent framework, under the assumption of the existence of a convergent supporting hyperplane. The following theorem is in this sense.

\begin{theorem}\label{thmmain1}
	Let $f \in  \Gamma_0(\cdot,\H)$, $x \in f^{-1}(\cdot,\R)$ and $v \in \partial f(0,x)$. Let us assume that $(\Delta^2_\t f(x | v))_{\t > 0}$ admits a convergent supporting hyperplane $(z_\t,\xi_\t,\beta_\t)_{\t > 0} \to (z,\xi,\beta)$. Then the following assertions are equivalent:
	\begin{enumerate}
		\item[{\rm (i)}] $f$ is twice epi-differentiable at $x$ for $v$;
		\item[{\rm (ii)}] $\partial f$ is proto-differentiable at $x$ for $v$ and $D_p (\partial f) (x | v)$ is a maximal monotone operator.
	\end{enumerate}
	In that case $d^2_e f(x|v) $ belongs to $\Gamma_0(\H)$ with $\beta \leq d^2_e f (x | v)(0) \leq 0$ and 
	$$ D_p (\partial f)(x|v) = \partial ( d^2_e f(x|v) ).$$
\end{theorem}

\begin{proof}
	Let us assume that $f$ is twice epi-differentiable at $x$ for $v$. Our first aim is to prove that $d^2_e f(x|v) \in \Gamma_0(\H)$. Firstly, let us assume by contradiction that there exists $w \in \H$ such that $d^2_e f (x | v)(w)=-\infty$. Let $\lambda \in \R$ such that $\lambda < \langle \xi , w \rangle + \beta$. Since $(w,\lambda) \in \epi (  d^2_e f (x | v) ) = \limsup \epi ( \Delta^2_\t f (x | v) )$, we know that there exist $(t_n) \to 0$ and $(w_n,\lambda_n) \to (w,\lambda)$ such that $(w_n,\lambda_n) \in \epi ( \Delta^2_{t_n} f (x | v) )$, that is, $ \Delta^2_{t_n} f (x | v) (w_n) \leq \lambda_n$ for all $n \in \N$. From Definition~\ref{defCSH}, we get that $ \langle \xi_{t_n} , w_n \rangle + \beta_{t_n} \leq \lambda_n$ for all $n \in \N$. Letting $n \to +\infty$ raises a contradiction. We prove that $d^2_e f (x | v)$ is lower semi-continuous, convex and that $d^2_e f (x | v) (0) \leq 0$ as in the proof of Proposition~\ref{proptinde}. Our second aim is to prove that $d^2_e f (x | v)(0) \geq \beta$. By contradiction, let us assume that $d^2_e f (x | v)(0) < \beta$. Then, there exists $\varepsilon > 0$ such that $(0,\beta-\varepsilon) \in \epi ( d^2_e f (x | v) ) = \limsup \epi ( \Delta^2_\t f (x | v) )$. Thus, there exist $(t_n) \to 0$ and $(w_n,\lambda_n) \to (0,\beta-\varepsilon)$ such that $(w_n,\lambda_n) \in \epi ( \Delta^2_{t_n} f (x | v) )$, that is, $ \Delta^2_{t_n} f (x | v) (w_n) \leq \lambda_n$ for all $n \in \N$. From Definition~\ref{defCSH}, we get that $ \langle \xi_{t_n} , w_n \rangle + \beta_{t_n} \leq \lambda_n$ for all $n \in \N$. Letting $n \to +\infty$ constitutes a contradiction. Our last aim is to prove that $\partial f$ is proto-differentiable at $x$ for $v$ with $ D_p (\partial f)(x|v) = \partial ( d^2_e f(x|v) )$. This result directly follows from Proposition~\ref{prop12} and Theorem~\ref{thmattouch}.
	
	\medskip
	
	Now let us assume that $\partial f$ is proto-differentiable at $x$ for $v$ and $D_p (\partial f) (x | v)$ is a maximal monotone operator. From Proposition~\ref{prop12}, we get that $( \partial ( \Delta^2_\t f (x | v) ) )_{\t > 0} = ( \Delta_\t (\partial f) (x | v) )_{\t > 0}$ G-converges to the maximal monotone operator $ D_p (\partial f)(x|v)$. We deduce from Theorem~\ref{thmattouch2} that $ D_p (\partial f)(x|v) = \partial \varphi$ for some $\varphi \in \Gamma_0(\H)$. Our aim is now to apply Theorem~\ref{thmattouch} in order to conclude that $( \Delta^2_\t f (x | v) ) _{\t > 0}$ M-converges. From Definition~\ref{defCSH}, we have $(z_\t,\xi_\t) \to (z,\xi)$ and since $\xi_\t \in \partial ( \Delta^2_\t f (x | v) ) (z_\t)$, we deduce that $(z_\t,\xi_\t) \in \Gr ( \partial ( \Delta^2_\t f (x | v) ) )$ for all $\t > 0$. From Remark~\ref{propgraphconv}, we deduce that $(z,\xi) \in \Gr (\partial \varphi)$ and thus $z \in \dom (\varphi)$. Considering $\psi := \varphi - \varphi (z) + \langle \xi , z \rangle + \beta \in \Gamma_0(\H)$, we have $\partial \psi = \partial \varphi$ and $\Delta^2_\t f(x|v)(z_\t) = \langle \xi_\t , z_\t \rangle + \beta_\t \to \psi (z)$ since $(z_\t,\xi_\t,\beta_\t)_{\t > 0}$ is a convergent supporting hyperplane of $(\Delta^2_\t f(x | v))_{\t > 0}$. The proof is thereby completed using Theorem~\ref{thmattouch}.
\end{proof}

\begin{remark}\normalfont
	In the $t$-independent framework, Theorem~\ref{thmmain1} exactly coincides with Proposition~\ref{proptinde}. Indeed, in that case, Lemma~\ref{lemtinde} ensures that  $(\Delta^2_\t f(x | v))_{\t > 0}$ admits a convergent supporting hyperplane with $(z_\t,\xi_\t,\beta_\t) = (0,0,0)$ for all $\t > 0$.
\end{remark}

\begin{remark}\normalfont
	In Example~\ref{excontreex}, one can easily see that $(\Delta^2_\t f(x | v))_{\t > 0}$ does not admit a convergent supporting hyperplane.
\end{remark}

\begin{remark}\normalfont
	In contrary to Proposition~\ref{proptinde}, it might be possible that $d^2_e f(x|v) (0) \neq 0$ in Theorem~\ref{thmmain1}. We refer to Example~\ref{excontreex2} below.
\end{remark}

\begin{example}\label{excontreex2}\normalfont
	Let $\H = \R$ and $f(t,x) := \vert x - t^2 \vert$ for all $(t,x) \in \R_+ \times \R$. Let us consider $x =0$ and $v = 0 \in \partial f (0,x)$. One can easily compute that
	$$ \Delta^2_\t f(x|v) (w) = \dfrac{ \vert w - \t \vert - \t }{\t}, $$
	for all $\t > 0$ and all $w \in \R$. One can easily deduce that $f$ is twice epi-differentiable at $x$ for $v$ with 
	$$ d^2_e f(x|v)(w) = \left\lbrace 
	\begin{array}{lcr}
	-1 & \text{if} & w = 0, \\
	+\infty & \text{if} & w \neq 0.
	\end{array}
	\right. $$
	In particular note that $d^2_e f(x|v)(0) \neq 0$ and thus $d^2_e f(x|v)$ is not positively homogeneous of degree two.
\end{example} 

We conclude this section by proving that the concept of \textit{convergent supporting hyperplane} is a necessary and sufficient condition in a sense given by the following proposition.

\begin{proposition}\label{propderniere}
	Let $f \in  \Gamma_0(\cdot,\H)$, $x \in f^{-1}(\cdot,\R)$ and $v \in \partial f(0,x)$. If $f$ is twice epi-differentiable at $x$ for $v$, then the following assertions are equivalent:
	\begin{enumerate}
		\item[{\rm (i)}] $d^2_e f(x|v) \in \Gamma_0(\H)$;
		\item[{\rm (ii)}] $(\Delta^2_\t f(x | v))_{\t > 0}$ admits a convergent supporting hyperplane.
	\end{enumerate}
\end{proposition}

\begin{proof}
	From Theorem~\ref{thmmain1}, we only need to prove the implication (i)$\Rightarrow$(ii). Let us assume that $d^2_e f(x|v) \in \Gamma_0(\H)$. From Theorem~\ref{thmattouch}, we deduce that there exists $(z_\t, \xi_\t)_{\t > 0} \to (z,\xi)$ such that $(z_\t , \xi_\t) \in \Gr ( \partial ( \Delta^2_\t f (x|v) ) )$ for all $\t > 0$ and $\Delta^2_\t f (x|v)(z_\t) \to  d^2_e f(x|v)(z)$. From Remark~\ref{remfinite}, we know that $d^2_e f(x|v)(z) \in \R$. This concludes the proof from Remark~\ref{remCSH}.

\end{proof}

\subsection{Proto-differentiability of the proximity operator and sensitivity analysis}\label{secmain2}

Let $(A,f) \in \mathcal{A}(\cdot,\H) \times \Gamma_0 (\cdot,\H)$. For the simplicity of notations we introduce
$$ \fonction{\Prox_{A,f}}{\R_+ \times \H}{\H}{(t,x)}{\Prox_{A,f} (t,x) := \prox_{A(t,\cdot),f(t,\cdot)} (x).} $$
Using the notations introduced at the beginning of Section~\ref{sec3}, note that $\Prox_{A,f} = (A+\partial f)^{-1}$. From Propositions~\ref{prop1} and \ref{prop2}, one can easily conclude the following lemma.

\begin{lemma}\label{lemlem}
	Let $(A,f) \in \mathcal{A}(\cdot,\H) \times \Gamma_0 (\cdot,\H)$ and $x \in \H$. Let us assume that $A$ is semi-differentiable at $v := \Prox_{A,f} (0,x)$. Then, $\Prox_{A,f}$ is proto-differentiable at $x$ for $v$ if and only if $\partial f$ is proto-differentiable at $v$ for $v_0 := x - A (0,v) \in \partial f ( 0 , v)$. In that case, it holds that
	$$ D_p (\Prox_{A,f}) (x|v) = \Big( D_s A (v) + D_p ( \partial f ) ( v | v_0 ) \Big)^{-1} .$$
\end{lemma} 

From Theorem~\ref{thmmain1}, Proposition~\ref{propderniere} and Lemma~\ref{lemlem}, we deduce the following theorem.

\begin{theorem}\label{thmmain2}
	Let $(A,f) \in \mathcal{A}(\cdot,\H) \times \Gamma_0 (\cdot,\H)$ and $x \in \H$. If the following assertions are satisfied:
	\begin{enumerate}
		\item[{\rm (i)}] $A \in \mathcal{A}_{\mathrm{unif}} (\cdot,\H)$;
		\item[{\rm (ii)}] $A$ is semi-differentiable at $v := \Prox_{A,f} (0,x)$;
		\item[{\rm (iii)}] $f$ is twice epi-differentiable at $v$ for $v_0 := x - A (0,v) \in \partial f ( 0 , v)$;
		\item[{\rm (iv)}] $d^2_e f(v|v_0) \in \Gamma_0(\H)$;
	\end{enumerate}
	then $\Prox_{A,f}$ is proto-differentiable at $x$ for $v$ with 
	$$ D_p (\Prox_{A,f}) (x|v) = \prox_{D_s A(v) , \, d^2_e f(v|v_0)} .$$
\end{theorem}

\begin{proof}
	From Proposition~\ref{propDA}, we know that $D_s A(v) \in \mathcal{A}(\H) $. Hence, the mapping $\prox_{D_s A(v), \, d^2_e f(v|v_0)} : \H \to \H$ is well-defined (see Section~\ref{secconvanal}). The proof of Theorem~\ref{thmmain2} easily follows from Theorem~\ref{thmmain1}, Proposition~\ref{propderniere} and Lemma~\ref{lemlem}.
\end{proof}

Now we return to the initial motivation of the present paper, that is, the sensitivity analysis, with respect to the parameter $t \geq 0$, of the general nonlinear variational inequality of second kind given by

\vspace{-0.19cm}
\begin{minipage}{.31\linewidth}
	\begin{equation*}
	\!\! \!\!\! \!\!\! (\mathrm{VI}( A(t,\cdot),f(t,\cdot),x(t) ))
	\end{equation*}
\end{minipage}
\begin{minipage}{.68\linewidth}
	\begin{equation*}
	\!\! \forall z \in \H, \;\; \langle A(t,y) , z - y \rangle + f(t,z) - f(t,y) \geq \langle x(t) , z - y \rangle ,
	\end{equation*}
\end{minipage}
\vspace{0.35cm}

where $(A,f) \in \mathcal{A}(\cdot,\H) \times \Gamma_0(\cdot,\H)$ and where $x : \R_+ \to \H$ is a given function. From Section~\ref{secconvanal}, the above variational inequality admits for all $t \geq 0$ a unique solution $y(t) \in \H$ given by
$$ y(t) = \prox_{A(t,\cdot),f(t,\cdot)} (x(t)) = \Prox_{A,f} (t,x(t)). $$
The next theorem is the major result of the present paper. It provides sufficient conditions on $A$, $f$ and $x$ under which $y : \R_+ \to \H$ is differentiable at $t=0$ and provides an explicit formula for~$y'(0)$.

\begin{theorem}\label{thmmain3}
	Let $(A,f) \in \mathcal{A}(\cdot,\H) \times \Gamma_0 (\cdot,\H)$ and let $x : \R_+ \to \H$ be a function. We consider the function $y : \R_+ \to \H$ defined by
	$$ y(t) := \prox_{A(t,\cdot),f(t,\cdot)} (x(t)) ,$$
	for all $t \geq 0$. If the following assertions are satisfied:
	\begin{enumerate}
		\item[{\rm (i)}] $x$ is differentiable at $t=0$;
		\item[{\rm (ii)}] $A \in \mathcal{A}_{\mathrm{unif}} (\cdot,\H)$;
		\item[{\rm (iii)}] $A$ is semi-differentiable at $y(0)$;
		\item[{\rm (iv)}] $f$ is twice epi-differentiable at $y(0)$ for $v_0 := x(0)-A(0,y(0)) \in \partial f ( 0 , y(0))$;
		\item[{\rm (v)}] $d^2_e f(y(0)|v_0) \in \Gamma_0(\H)$;
	\end{enumerate}
	then $y : \R_+ \to \H$ is differentiable at $t=0$ with 
	$$ y'(0) = \prox_{D_s A( y(0) ), \, d^2_e f ( y(0) | v_0 )} (x'(0)) .$$
\end{theorem}

\begin{proof}
	For the ease of notations, we denote by $\prox_{\chi} := \prox_{D_s A( y(0) ), \, d^2_e f ( y(0) | v_0 )}  $. Recall that $\prox_\chi$ is $\frac{1}{\alpha}$-Lipschitz continuous where $\alpha > 0$ denotes the uniform strong monotonicity coefficient of $A \in \mathcal{A}_{\mathrm{unif}} (\cdot,\H)$ (see Section~\ref{secconvanal} and Remark~\ref{remunifconstant}). From Theorem~\ref{thmmain2}, we deduce that $\Prox_{A,f}$ is proto-differentiable at $x(0)$ for $y(0)$ with
	$$ D_p (\Prox_{A,f}) (x(0)|y(0)) = \prox_{\chi} .$$
	By contradiction, let us assume that there exist $\varepsilon > 0$ and $(t_n) \to 0$ such that 
	$$ \varepsilon \leq \left\Vert \frac{y(t_n)-y(0)}{t_n} - \prox_{\chi}( x'(0) ) \right\Vert ,$$
	for all $n \in \N$. Since $\prox_{\chi} (x'(0)) = D_p ( \Prox_{A,f}) (x(0)|y(0)) (x'(0))$, we deduce that there exist $(w_n,z_n)_{n \in \N} \to (x'(0),\prox_{\chi} (x'(0)))$ and $N \in \N$ such that 
	$$ (w_n,z_n) \in \Gr \Big( \Delta_{t_n} (\Prox_{A,f}) (x(0)|y(0)) \Big) $$ 
	for all $n \geq N$. We deduce that 
	$$ \prox_{A(t_n,\cdot),f(t_n,\cdot)} ( x(0) + t_n w_n ) = y(0) + t_n z_n, $$
	for all $n \geq N$. Finally, we obtain that
	\begin{multline*}
	\varepsilon \leq \left\Vert \dfrac{ \prox_{A(t_n,\cdot),f(t_n,\cdot)} (x(t_n))  - \prox_{A(t_n,\cdot),f(t_n,\cdot)} ( x(0) + t_n w_n )  }{t_n} \right\Vert \\ 
	+ \left\Vert \dfrac{ \prox_{A(t_n,\cdot),f(t_n,\cdot)} ( x(0) + t_n w_n ) - y(0) }{t_n} - z_n \right\Vert + \Vert z_n - \prox_\chi ( x'(0) ) \Vert ,
	\end{multline*}
	for all $n \in \N$. The second term is equal to zero for all $n \geq N$. Using the $\frac{1}{\alpha}$-Lipschitz continuity of $\prox_{A(t_n,\cdot),f(t_n,\cdot)}$, we obtain that $ \varepsilon \leq \frac{1}{\alpha} \Vert \frac{x(t_n)-x(0)}{t_n} - w_n \Vert + \Vert z_n - \prox_\chi ( x'(0) ) \Vert$ for all $n \geq N$. This raises a contradiction since $w_n \to x'(0)$ and $z_n \to \prox_{\chi} (x'(0))$. The proof is complete.
\end{proof}

\begin{remark}\normalfont
	Let us assume that all assumptions of Theorem~\ref{thmmain3} are satisfied. Then its conclusion can be rewritten as follows. If $y(t)$ is the unique solution of the variational inequality $(\mathrm{VI}(A(t,\cdot), \, f(t,\cdot), \, x(t)))$ for all $t \geq 0$, then $y : \R_+ \to \H$ is differentiable at $t=0$ and $y'(0)$ is the unique solution of $(\mathrm{VI}(D_s A(y(0)), \, d^2_e f (y(0)|v_0), \, x'(0)))$.
\end{remark}

\section{Applications to parameterized convex minimization problems}\label{sec5}

In this section our aim is to apply our main result (Theorem~\ref{thmmain3}) to the sensitivity analysis of parameterized convex minimization problems. We start with a general result (see Proposition~\ref{propgenefin} in Section~\ref{secappl1}), then we give a simple illustration in the one-dimensional setting (see Section~\ref{secappl2}). We conclude by studying in details the case of parameterized smooth convex minimization problems with inequality constraints in a finite-dimensional setting (see Proposition~\ref{propfinalfinal} in Section~\ref{secappl3}).

\subsection{A general result}\label{secappl1}
In this section we will prove from Theorem~\ref{thmmain3} that the derivative of the solution of a parameterized convex minimization problem is still, under some appropriate assumptions, the solution of a convex minimization problem. The following proposition is in this sense.

\begin{proposition}\label{propgenefin}
	Let $f \in \Gamma_0(\cdot,\H)$ and $\ell : \R_+ \to \H$ be given functions. Let $g : \R_+ \times \H \to \R$ be such that $g(t,\cdot)$ is differentiable on $\H$ with $\nabla_x g(t,\cdot) \in \mathcal{A}(\H)$ for all $t \geq 0$. Then, for all $t \geq 0$, the parameterized convex minimization problem
	
	\vspace{-0.19cm}\noindent
	\begin{minipage}{.1\linewidth}
		\begin{equation*}
		(\mathcal{M}_1 (t))
		\end{equation*}
	\end{minipage}
	\begin{minipage}{.8\linewidth}
		\begin{equation*}
		\argmin_{x \in \H} \Big[ f(t,x) + g(t,x) - \langle \ell(t) , x \rangle \Big]
		\end{equation*}
	\end{minipage}\vspace{0.35cm}
	
	admits a unique solution denoted by $y(t)$. If moreover the following assumptions are satisfied:
	\begin{enumerate}
		\item[{\rm (i)}] $\ell$ is differentiable at $t=0$;
		\item[{\rm (ii)}] $\nabla_x g \in \mathcal{A}_{\mathrm{unif}}(\cdot,\H)$;
		\item[{\rm (iii)}] $\nabla_x g$ is of class $C^1$ on $\R_+ \times \H$;
		\item[{\rm (iv)}] $f$ is twice epi-differentiable at $y(0)$ for $v_0 := \ell (0)-\nabla_x g(0,y(0)) \in \partial f ( 0 , y(0))$;
		\item[{\rm (v)}] $d^2_e f (y(0),v_0) \in \Gamma_0(\H)$;
	\end{enumerate}
	then $y : \R_+ \to \H$ is differentiable at $t=0$ and $y'(0)$ is the unique solution of the convex minimization problem given by
	
	\vspace{-0.19cm}\noindent
	\begin{minipage}{.1\linewidth}
		\begin{equation*}
		(\mathcal{M}_1'(0))
		\end{equation*}
	\end{minipage}
	\begin{minipage}{.8\linewidth}
		\begin{multline*}
		\qquad \argmin_{x \in \H} \Big[ d^2_e f (y(0),v_0) (x) + \dfrac{1}{2} \langle \nabla^2_{xx} g(0, y(0) ) (x) , x \rangle \\
		\qquad + \langle \nabla^2_{tx} g (0 , y(0) ) -  \ell'(0) , x \rangle \Big] . 
		\end{multline*}
	\end{minipage}\vspace{0.35cm}	
	
\end{proposition}

\begin{proof}
	Since $\nabla_x g (t,\cdot)$ is monotone, we deduce that $g(t,\cdot)$ is convex on $\H$ for all $t \geq 0$. As a consequence, for all $t \geq 0$, $y(t)$ is a solution of the convex minimization problem~$(\mathcal{M}_1(t))$ if and only if
	$$ 0 \in \partial \Big( f(t,\cdot)+g(t,\cdot) - \langle \ell(t) , \cdot \rangle \Big) (y(t)) = \partial f(t,y(t))+\partial g(t,y(t)) - \ell(t) . $$
	In the above equation, since $g(t,\cdot)$ is differentiable and $\langle \ell(t),\cdot \rangle$ is linear continuous, the subdifferential of the sum is equal to the sum of the subdifferentials. Denoting by $A(t,\cdot) := \partial g(t,\cdot) = \nabla_x g (t,\cdot) \in \mathcal{A}(\H)$, one can easily conclude that the convex minimization problem~$(\mathcal{M}_1(t))$ admits a unique solution $y(t)$ given by
	$$ y(t) = \prox_{A(t,\cdot),f(t,\cdot)} (\ell(t)). $$	
	Let us prove the second part of Proposition~\ref{propgenefin}. Since $\nabla_x g$ is of class $C^1$ on $\R_+ \times \H$, one can easily prove that $A$ is semi-differentiable on $\H$ with
	$$ D_s A(x)(w) = \nabla^2_{xx} g(0,x)(w) + \nabla^2_{tx} g(0,x), $$
	for all $x$, $w \in \H$. From Theorem~\ref{thmmain3}, we deduce that $y : \R_+ \to \H$ is differentiable at $t=0$ with
	$$ y'(0) = \prox_{D_s A( y(0) ), \, d^2_e f ( y(0) | v_0 )} (\ell'(0)) .$$
	We deduce that
	$$ \ell'(0) \in D_s A( y(0) ) ( y'(0) ) + \partial \Big( d^2_e f ( y(0) | v_0 ) \Big) (y'(0)) . $$
	Since the function $ x \in \H \mapsto \frac{1}{2} \langle \nabla^2_{xx} g(0,y(0))(x) , x \rangle \in \R$ is convex and differentiable on $\H$, we obtain that
	$$ 0 \in \partial \Big( d^2_e f ( y(0) | v_0 ) + \dfrac{1}{2} \langle \nabla^2_{xx} g(0,y(0))(\cdot) , \cdot \rangle + \langle \nabla^2_{tx} g(0,y(0)) - \ell'(0) , \cdot \rangle \Big) (y'(0)) ,$$
	which concludes the proof.
\end{proof}

\subsection{Illustration with a one-dimensional example}\label{secappl2}
In this section we denote by $x^+ := \max (0,x)$ and by $x^- := \min (0,x)$ for all $x \in \R$.

\medskip

In order to illustrate Proposition~\ref{propgenefin}, we study in this section a one-dimensional example. Precisely we consider the parameterized one-dimensional convex minimization problem given by

\vspace{-0.19cm}\noindent
\begin{minipage}{.1\linewidth}
	\begin{equation*}
	(\mathcal{M}_2 (t))
	\end{equation*}
\end{minipage}
\begin{minipage}{.8\linewidth}
	\begin{equation*}
	\argmin_{x \in \R} \Big( a(t) \vert x - b(t) \vert + \frac{c(t)}{2} x^2 - d(t) x \Big), 
	\end{equation*}
\end{minipage}\vspace{0.35cm}

where 
\begin{itemize}
	\item $a : \R_+ \to \R$ is differentiable at $t=0$ and $a(t) > 0$ for all $t \geq 0$;
	\item $b : \R_+ \to \R$ is twice differentiable at $t=0$ and $b'(0)=0$;
	\item $c : \R_+ \to \R$ is of class $C^1$ on $\R_+$ and $\alpha \leq c(t) \leq \beta $ for all $t \geq 0$ for some $0 < \alpha \leq \beta$;
	\item $d : \R_+ \to \R$ is differentiable at $t=0$.
\end{itemize}
In order to apply Proposition~\ref{propgenefin}, one has to consider $\H = \R$ and 
$$ f(t,x) := a(t) \vert x - b(t) \vert,  \quad g(t,x) := \frac{c(t)}{2} x^2 \quad \text{and} \quad \ell(t):=d(t) , $$
for all $(t,x) \in \R_+ \times \R$. Almost all hypotheses of Proposition~\ref{propgenefin} are easily checkable. Actually the only difficult part is to check that $f$ is twice epi-differentiable and that its second epi-derivative belongs to $\Gamma_0(\H)$. This follows from the following lemma.

\begin{lemma}\label{lemfin}
	Let $f : \R_+ \times \R \to \R$ be defined by $f(t,x) := a(t) \vert x - b(t) \vert$ for all $(t,x) \in \R_+ \times \R$, where 
	\begin{itemize}
		\item $a : \R_+ \to \R$ is differentiable at $t=0$ and $a(t) > 0$ for all $t \geq 0$;
		\item $b : \R_+ \to \R$ is twice differentiable at $t=0$ and $b'(0)=0$.
	\end{itemize}
	Then $f$ is twice epi-differentiable at every $x \in \R$ for every $v \in \partial f(0,x)$ and we obtain five different cases:
	\begin{enumerate}
		\item[{\rm (i)}] If $x > b(0)$, then $v=a(0)$ and we obtain that $d^2_e f(x|v)(w) = a'(0)w$ for all $w \in \R$;
		\item[{\rm (ii)}] If $x=b(0)$, then $v \in [-a(0),a(0)]$. 
		\begin{enumerate}
			\item If $v=a(0)$, we obtain that
			$$ d^2_e f(x|v)(w) = \left\lbrace 
			\begin{array}{lcr}
			a'(0)w - a(0) (b''(0))^+ & \text{if} & w \geq 0, \\
			+\infty & \text{if} & w < 0,
			\end{array}
			\right. $$
			for all $w \in \R$.
			\item If $v \in (-a(0),a(0))$, we obtain that
			$$ d^2_e f(x|v)(w) = \left\lbrace 
			\begin{array}{lcr}
			\left( \dfrac{a(0)-v}{2} \right) b''(0) - a(0) ( b''(0) )^+ & \text{if} & w = 0, \\
			+\infty & \text{if} & w \neq 0,
			\end{array}
			\right. $$
			for all $w \in \R$.
			\item If $v=-a(0)$, we obtain that
			$$ d^2_e f(x|v)(w) = \left\lbrace 
			\begin{array}{lcr}
			- a'(0)w - a(0) (-b''(0))^+ & \text{if} & w \leq 0, \\
			+\infty & \text{if} & w > 0,
			\end{array}
			\right. $$
			for all $w \in \R$.
		\end{enumerate}
		\item[{\rm (iii)}] If $x < b(0)$, then $v=-a(0)$ and we obtain that $d^2_e f(x|v)(w) = - a'(0)w$ for all $w \in \R$.
	\end{enumerate}
	In each above case, it holds that $d^2_e f(x|v) \in \Gamma_0(\H)$.
\end{lemma}

\begin{proof}
	One can easily compute that
	$$ \Delta^2_\t f(x|v)(w) = \left\lbrace 
	\begin{array}{lcr}
	\left( \dfrac{a(\t)-v}{\t} \right) w - 2 a(\t) \left( \dfrac{b(\t)-x}{\t^2} \right)^+ & \text{if} & w \geq \dfrac{b(\t)-x}{\t}, \\[10pt]
	- \left( \dfrac{a(\t)+v}{\t} \right) w - 2 a(\t) \left( \dfrac{x-b(\t)}{\t^2} \right)^+ & \text{if} & w \leq \dfrac{b(\t)-x}{\t},
	\end{array}
	\right. $$
	for all $\t > 0$, all $(w,x) \in \R \times \R$ and all $v \in \partial f(0,x)$. Lemma~\ref{lemfin} follows.
\end{proof}

\begin{remark}\normalfont
	Note that Lemma~\ref{lemfin} encompasses the result of Example~\ref{excontreex2}. In Example~\ref{excontreex}, all assumptions of Lemma~\ref{lemfin} are satisfied except $b'(0) = 0$.
\end{remark}

From Proposition~\ref{propgenefin} and Lemma~\ref{lemfin}, we obtain that, for all $t \geq 0$, the convex minimization problem~$(\mathcal{M}_2(t))$ admits a unique solution denoted by $y(t)$. Moreover, we know that $y : \R_+ \to \H$ is differentiable at $t=0$ and that $y'(0)$ is solution of the convex minimization problem

\vspace{-0.19cm}\noindent
\begin{minipage}{.18\linewidth}
	\begin{equation*}
	\!\!\! \!\!\! \!\!\! \!\!\! \!\!\! \!\!\! (\mathcal{M}_2'(0))  
	\end{equation*}
\end{minipage}
\begin{minipage}{.8\linewidth}
	\begin{equation*}
	\argmin_{x \in \R} \Big[ d^2_e f \Big( y(0) | d(0) - c(0)y(0) \Big) (x) + \dfrac{c(0)}{2} x^2 + (c'(0)y(0)-d'(0))x \Big] . 
	\end{equation*}
\end{minipage}\vspace{0.35cm}

At this step recall that $d (0)-c(0)y(0) \in \partial f ( 0 , y(0))$. Thus we get three cases:
\begin{enumerate}
	\item[{\rm (i)}] if $y(0) > b(0)$, then $d(0)-c(0)y(0)=a(0)$;
	\item[{\rm (ii)}] if $y(0) = b(0)$, then $d(0)-c(0)y(0) \in [-a(0),a(0)]$;
	\item[{\rm (iii)}] if $y(0) < b(0)$, then $d(0)-c(0)y(0)=-a(0)$.
\end{enumerate}
Finally, from Lemma~\ref{lemfin} and $(\mathcal{M}'_2(0))$, we obtain that
$$ y'(0) = \left\lbrace 
\begin{array}{ccl}
\left( \dfrac{d-a}{c} \right)' (0) & \text{if} & y(0) > b(0) , \\[10pt]
\left( \left( \dfrac{d-a}{c} \right)' (0) \right)^+ & \text{if} & y(0) = b(0) = \dfrac{d(0)-a(0)}{c(0)} , \\[10pt]
0 & \text{if} & y(0) = b(0) \in \left( \dfrac{d(0)-a(0)}{c(0)} , \dfrac{d(0)+a(0)}{c(0)} \right) , \\[10pt]
\left( \left( \dfrac{d+a}{c} \right)' (0) \right)^- & \text{if} & y(0) = b(0) = \dfrac{d(0)+a(0)}{c(0)} , \\[10pt]
\left( \dfrac{d+a}{c} \right)' (0) & \text{if} & y(0) < b(0).
\end{array}
\right. $$
The above results are perfectly coherent since it can be proved that the unique solution~$y(t)$ of the parameterized convex minimization problem~$(\mathcal{M}_2(t))$ is given by
$$ y(t) = \left\lbrace 
\begin{array}{ccl}
\dfrac{d(t)-a(t)}{c(t)} & \text{if} & b(t) \leq \dfrac{d(t)-a(t)}{c(t)} , \\[10pt]
b(t)  & \text{if} & \dfrac{d(t)-a(t)}{c(t)} < b(t) < \dfrac{d(t)+a(t)}{c(t)}, \\[10pt]
\dfrac{d(t)+a(t)}{c(t)} & \text{if} & \dfrac{d(t)+a(t)}{c(t)} \leq b(t),
\end{array}
\right. $$
for all $t \geq 0$.

\subsection{Applications to parameterized smooth convex minimization problems with inequality constraints}\label{secappl3}
Let $m$, $d \in \N^*$. In the whole section we denote by $\R_-$ the set of all nonpositive real numbers and by $\R^d_- := \R_- \times \ldots \times \R_-$.

\medskip

In this section we focus on a general finite-dimensional parameterized convex minimization problem with inequality constraints given by
$$
\argmin_{\substack{ ~\\ x \in \R^m \\[2pt] F (t,x) \in \R^d_- }} g(t,x) 
$$
where $F := (F_i)_{i=1,\ldots,d} : \R_+ \times \R^m \to \R^d$ and $g : \R_+ \times \R^m \to \R$ are smooth functions.

\medskip

Precisely, under some appropriate assumptions, we will prove in Proposition~\ref{propfinalfinal} that the above minimization problem admits, for all $t \geq 0$, a unique solution denoted by $y(t)$, and that the function $y : \R_+ \to \R^m$ is differentiable at $t=0$ where $y'(0)$ is the unique solution of a convex minimization problem with linear inequality/equality constraints. We refer to Proposition~\ref{propfinalfinal} and Remarks~\ref{remfin} and \ref{remfin2} for details.

\medskip

To this aim, we need to recall first some classical notions and notations. Let $\mathrm{N} ( w )$ (resp. $\mathrm{T} ( w )$) denote the \textit{normal cone} (resp. \textit{tangent cone}) to $\R^d_-$ at some $w \in \R^d_-$. Recall that $\mathrm{N} ( w )$ and $\mathrm{T} ( w )$ are both closed convex cones of $\R^d$ containing $0_{\R^d}$, and that they are mutually polar. Let $\delta_{\mathrm{Y}} \in \Gamma_0(\R^d)$ (resp. $\sigma_{\mathrm{Y}}  \in \Gamma_0(\R^d)$) denote the \textit{indicator function} (resp. \textit{support function}) of a nonempty closed convex subset $\mathrm{Y} \subset \R^d_-$. We refer to standard books like~\cite{JBHU,Penot,RockWets} and references therein for details. 

\medskip

If $F$ is of class $C^2$ on $\R_+ \times \R^m$, we denote by $D^2 F(t,x)$ the classical second-order Fr\'echet differential of $F$ at $(t,x) \in \R_+ \times \R^m$. We also introduce the following sets:
$$ \mathrm{K} ( y | v ) := \{ x \in \R^m \mid  \nabla_x F(0,y) x \in \mathrm{T} ( F(0,y) ) \text{ and } \langle x , v \rangle = 0 \} $$
and
$$ \mathrm{Y} ( y | v ) := \{ w \in \R^d \mid w \in \mathrm{N} ( F(0,y) )  \text{ and }  \nabla_x F(0,y)^\top w = v \} ,$$
for all $(y,v) \in \R^m \times \R^m$ such that $F(0,y) \in \R^d_-$.

\medskip

We are now in position to state and prove the main result of this section.

\begin{proposition}\label{propfinalfinal}
	Let $m$, $d \in \N^*$. Let $F := (F_i)_{i=1,\ldots,d} : \R_+ \times \R^m \to \R^d$ be such that $F_i \in \Gamma_0(\cdot,\R^m)$ for every $i=1,\ldots,d$. We assume that $C(t) := \{ x \in \R^m \mid F(t,x) \in \R^d_- \}$ is not empty for all $t \geq 0$. Let $g : \R_+ \times \R^m \to \R$ be such that $g(t,\cdot)$ is differentiable on $\R^m$ with $\nabla_x g(t,\cdot) \in \mathcal{A}(\R^m)$ for all $t \geq 0$. Then, for all $t \geq 0$, the parameterized convex minimization problem with inequality constraints
	
	\vspace{-0.19cm}\noindent
	\begin{minipage}{.1\linewidth}
		\begin{equation*}
		(\mathcal{M}_3 (t))
		\end{equation*}
	\end{minipage}
	\begin{minipage}{.8\linewidth}
		\begin{equation*}
		\argmin_{\substack{ ~\\ x \in \R^m \\[2pt] F (t,x) \in \R^d_- }} g(t,x)  
		\end{equation*}
	\end{minipage}\vspace{0.35cm}	
	
	admits a unique solution denoted by $y(t)$. If moreover the following assumptions are satisfied:
	\begin{enumerate}
		\item[{\rm (i)}] $\nabla_x g \in \mathcal{A}_{\mathrm{unif}}(\cdot,\R^m)$;
		\item[{\rm (ii)}] $\nabla_x g$ is of class $C^1$ on $\R_+ \times \R^m$;
		\item[{\rm (iii)}] $F$ is of class $C^2$ on $\R_+ \times \R^m$;
		\item[{\rm (iv)}] $\nabla_t F(0,y(0)) = 0_{\R^d}$;
		\item[{\rm (v)}] $y(0) \in C(t)$ for all $t \geq 0$;
		\item[{\rm (vi)}] $\Vert \nabla_x F(t,y(0))^\top w \Vert_{\R^m} \geq \alpha$ for all $w \in \mathrm{N} ( F(t,y(0)) )$ with $\Vert w \Vert_{\R^d} = 1$ and all $t \geq 0$, for some $\alpha > 0$;
	\end{enumerate}
	then $y : \R_+ \to \R^m$ is differentiable at $t=0$ and $y'(0)$ is the unique solution of the convex minimization problem with linear inequality/equality constraints given by
	
	\vspace{-0.19cm}\noindent
	\begin{minipage}{.1\linewidth}
		\begin{equation*}
		(\mathcal{M}_3 '(0))
		\end{equation*}
	\end{minipage}
	\begin{minipage}{.8\linewidth}
		\begin{multline*}
		\argmin_{\substack{ ~\\ x \in \R^m \\[3pt] x \in \mathrm{K} ( y(0) | v_0) }} \left[ \dfrac{1}{2} \langle \nabla^2_{xx} g(0, y(0) ) (x) , x \rangle + \langle \nabla^2_{tx} g (0 , y(0) ) , x \rangle \right. \\
		\left. + \sigma_{\mathrm{Y} (y(0) | v_0)} \left( \dfrac{1}{2} D^2 F(0,y(0)) (1,x) \right) \right] ,
		\end{multline*}
	\end{minipage}\vspace{0.35cm}		
	
	%
	where $v_0 := -\nabla_x g(0,y(0))$.
\end{proposition}

\begin{proof}
	Let us prove the first part of Proposition~\ref{propfinalfinal}. In order to apply the first part of Proposition~\ref{propgenefin} with $\H = \R^m$, we introduce $f : \R_+ \times \R^m \to \R \cup \{ +\infty \}$ defined by $f(t,x) := \delta_{C(t)} (x) = \delta_{\R^d_-} (F(t,x))$ for all $(t,x) \in \R_+ \times \R^m$. Thus, the minimization problem~$(\mathcal{M}_3(t))$ can be rewritten as 
	$$
	\argmin_{x \in \R^m } \Big[ f(t,x) + g(t,x) \Big].
	$$
	Since $F_i \in \Gamma_0(\cdot,\R^m)$ for every $i=1,\ldots,d$ and since $C(t)$ is not empty for all $t \geq 0$, one can easily deduce that $f \in \Gamma_0 (\cdot,\R^m)$ and thus the first part of Proposition~\ref{propgenefin} can be applied.
	
	\medskip
	
	Now let us prove the second part of Proposition~\ref{propfinalfinal}. In order to apply the second part of Proposition~\ref{propgenefin} with $\H = \R^m$, we only need to prove that $f$ is twice epi-differentiable at $y(0)$ for $v_0$ and that $d^2_e f (y(0),v_0) \in \Gamma_0(\R^m)$. To do so, we will essentially adapt the proof of \cite[Theorem~13.14 p.594]{RockWets} to the $t$-dependent framework. 
	
	\medskip
	
	First of all, we need to state some assertions:
	\begin{itemize}
		\item Note that $\mathrm{Y} ( y(0) | v_0 )$ is a closed convex subset of $\R^d$. Moreover, since $v_0 = -\nabla_x g(0,y(0)) \in \partial f(0,y(0)) = \nabla_x F(0,y(0))^\top \mathrm{N} ( F(0,y(0)) )$, we deduce that $\mathrm{Y} ( y(0) | v_0 )$ is nonempty. Finally, since $\Vert \nabla_x F(0,y(0))^\top x \Vert_{\R^m} \geq \alpha$ for all $x \in \mathrm{N} ( F(0,y(0)) )$ with $\Vert x \Vert_{\R^d} = 1$, one can easily prove that $\mathrm{Y} ( y(0) | v_0 )$ is compact. In particular, it follows that $\sigma_{\mathrm{Y} ( y(0) | v_0 ) } (x) = \max_{w \in \mathrm{Y} ( y(0) | v_0 )} \langle w , x \rangle_{\R^d} < +\infty$ for all $x \in \R^d$.
		\item Note that $\mathrm{K} ( y(0) | v_0 )$ is a closed convex cone of $\R^m$ containing $0_{\R^m}$. Then we can easily deduce that the function $\varphi : \R^m \to \R \cup \{ +\infty \}$, defined by 
		$$ \varphi (x) := \delta_{ \mathrm{K} ( y(0) | v_0 )} (x) + \sigma_{\mathrm{Y} (y(0) | v_0)} \left( \dfrac{1}{2} D^2 F(0,y(0)) (1,x) \right) ,$$ 
		for all $x \in \R^m$, is such that $\varphi \in \Gamma_0( \R^m)$. In particular we used the fact that each component of some $w \in \mathrm{Y} (y(0) | v_0)$ are nonnegative.
		\item For all $w \in \mathrm{Y} ( y(0) | v_0 )$, $\delta_{\R^d_-}$ is twice epi-differentiable at $F(0,y(0))$ for $w$ and it holds that
		$$ d^2_e \delta_{\R^d_-} (F(0,y(0)) | w ) ( \nabla_x F(0,y(0)) x) = \delta_{ \mathrm{K} ( y(0) | v_0 )} (x) ,$$
		for all $x \in \R^m$. We refer to \cite[Exercice~13.17 p.600]{RockWets} for details. In particular, note that the above formula is independent of the choice of $w \in \mathrm{Y} ( y(0) | v_0 )$.
	\end{itemize}
	Our aim from now is to prove that $f$ is twice epi-differentiable at $y(0)$ for $v_0$ with $d^2_e f (y(0)|v_0) = \varphi$ from the characterization of epi-convergence recalled in Proposition~\ref{propcharact} (with $\H = \R^m$ is finite-dimensional). Firstly, let $z \in \R^m$ and $(z_\t)_{\t > 0} \to z$. Adapting the proof of \cite[Theorem~13.14 p.594]{RockWets} to the $t$-dependent framework, one can easily show that
	$$ \Delta^2_\t f(y(0)|v_0)(z_\t) = \Delta^2_\t \delta_{\R^d_-} ( F(0,y(0)) | w ) ( \tilde{\Delta}_\t F ) + \left\langle w , \frac{\t \tilde{\Delta}_\t F - \t \nabla_x F(0,y(0)) z_\t }{\t^2} \right\rangle,  $$
	for any $w \in \mathrm{Y} ( y(0) | v_0 )$, where $\tilde{\Delta}_\t F := \frac{F(\t,y(0)+\t z_\t)-F(0,y(0))}{\t}$. Since $\nabla_t F(0,y(0)) = 0_{\R^d}$, one has $\tilde{\Delta}_\t F \to \nabla_x F(0,y(0)) z$. Moreover, since $\delta_{\R^d_-}$ is twice epi-differentiable at $F(0,y(0))$ for $w$ and since $\frac{\t \tilde{\Delta}_\t F - \t \nabla_x F(0,y(0)) z_\t }{\t^2} \to \frac{1}{2} D^2 F(0,y(0)) (1,z)$, we deduce that 
	\begin{multline*}
	\liminf_{\t \to 0} \Delta^2_\t f(y(0)|v_0)(z_\t) \geq d^2_e \delta_{\R^d_-} (F(0,y(0)) | w ) ( \nabla_x F(0,y(0)) z) \\
	+ \frac{1}{2} \langle w , D^2 F(0,y(0)) (1,z) \rangle .
	\end{multline*}	
	Since $d^2_e \delta_{\R^d_-} (F(0,y(0)) | w ) ( \nabla_x F(0,y(0)) z) =  \delta_{ \mathrm{K} ( y(0) | v_0 )} (z)$ and since the last inequality is satisfied for any $w \in \mathrm{Y} ( y(0) | v_0 )$ compact, we get that $\liminf_{\t \to 0} \Delta^2_\t f(y(0)|v_0)(z_\t) \geq \varphi (z)$.
	
	\medskip
	
	Secondly, our objective is now to exhibit a sequence $(z_\t)_{\t > 0} \to z$ that satisfies $\limsup_{\t \to 0} \Delta^2_\t f(y(0)|v_0)(z_\t) \leq \varphi (z)$. If $\varphi (z) = +\infty$, nothing has to be proved. Thus we assume from now that $\varphi (z) < +\infty$, that is, $z \in \mathrm{K} ( y(0) | v_0 )$. It is sufficient to construct a sequence $(z_\t)_{\t > 0} \to z$ such that 
	$$ \lim_{\t \to 0} \Delta^2_\t f(y(0)|v_0)(z_\t) = \frac{1}{2} \langle \overline{w} , D^2 F(0,y(0))(1,z) \rangle $$ 
	where $ \overline{w} \in \argmax_{w \in \mathrm{Y} ( y(0) | v_0 )} \langle w , D^2 F(0,y(0))(1,z) \rangle $. From \cite[Theorem~11.42 and Example~11.43 p.506]{RockWets}, we deduce the existence of $\overline{z} \in \R^m$ such that $-\langle v_0 , \overline{z} \rangle = \langle \overline{w} , D^2 F(0,y(0))(1,z) \rangle$ and $\eta := \nabla_x F(0,y(0)) \overline{z} + D^2 F(0,y(0))(1,z) \in \mathrm{T} (F(0,y(0)))$. Since $ \eta \in \mathrm{T} (F(0,y(0)))$ and $\nabla_x F(0,y(0)) z \in \mathrm{T} (F(0,y(0)))$ and from \cite[Theorem~13.11 and Proposition~13.12 p.591-592]{RockWets}, there exists $\lambda : \R_+ \to \R^d$ such that $\lambda (t) \in \R^d_-$ for all $t \geq 0$ and $\lambda(0) = F(0,y(0))$, $\lambda ' (0) = \nabla_x F(0,y(0)) z$ and $\lambda ''(0) = \eta$. On the other hand, from Assumption~(vi) and from the metric regularity property (see \cite[Theorem~9.43 and Example~9.44 p.387-388]{RockWets}), we know that
	$$ \mathrm{d} \left( y(0)+tz+ \dfrac{t^2}{2} \overline{z} , C(t) \right) \leq \frac{1}{\alpha} \mathrm{d} \left( F \left( t, y(0)+tz+ \dfrac{t^2}{2}  \overline{z} \right) , \R^d_- \right), $$
	for sufficiently small $t \geq 0$. Thus we deduce that
	\begin{multline*}
	\mathrm{d} \left( \overline{z} , \dfrac{C(t)- y(0)-t z}{t^2 / 2} \right) \leq \frac{1}{\alpha} \mathrm{d} \left( \tilde{\Delta}^2_t F  , \dfrac{\R^d_- - F(0,y(0)) - t \nabla_x F(0,y(0)) z}{t^2 /2} \right) \\
	\leq \frac{1}{\alpha} \left\Vert \tilde{\Delta}^2_t F - \dfrac{\lambda(t) - \lambda(0)-t \lambda '(0) }{t^2 / 2} \right\Vert_{\R^d}
	\end{multline*}
	for sufficiently small $t \geq 0$, where $\tilde{\Delta}^2_t F := \frac{F(t,y(0)+tz+ \frac{t^2}{2} \overline{z}) - F(0,y(0)) - t \nabla_x F(0,y(0)) z}{t^2 / 2}$. Since $\tilde{\Delta}^2_t F \to \eta = \lambda ''(0)$, we deduce that there exists $\xi : \R_+ \to \R^m$ such that $\xi (t) \in C(t)$ for all $t \geq 0$ and $\xi (0) = y(0)$, $\xi '(0) = z$ and $\xi ''(0) = \overline{z}$. Finally, we define $z_\t := \frac{\xi (\t) - \xi (0)}{\t} \to \xi '(0) = z$. Since $F(t,y(0)) \in \R^d_-$ and $F(t,\xi(t)) \in \R^d_-$ for all $t \geq 0$ and since $\langle z , v_0 \rangle =0$, we get that
	$$ \Delta^2_\t f(y(0)|v_0)(z_\t) = - \left\langle v_0 , \dfrac{z_\t - z}{\t} \right\rangle \longrightarrow - \dfrac{1}{2} \langle v_0 , \overline{z} \rangle = \dfrac{1}{2} \langle \overline{w} , D^2 F(0,y(0))(1,z) \rangle , $$
	which concludes the proof.
\end{proof}

\begin{remark}\label{remfin}\normalfont
	Note that the constraints $x \in \mathrm{K} (y(0)|v_0)$ in the convex minimization problem~$(\mathcal{M}'_3(0))$ can be written as linear inequality/equality constraints.
\end{remark}

\begin{remark}\label{remfin2}\normalfont
	Let us assume that all hypotheses of Proposition~\ref{propfinalfinal} are satisfied. In that case, recall that $\mathrm{Y} ( y(0) | v_0 )$ is a nonempty compact convex subset of $\R^d$ and thus
	$$ \sigma_{\mathrm{Y} ( y(0) | v_0 )} \left( \dfrac{1}{2} D^2 F(0,y(0)) (1,x) \right) = \dfrac{1}{2} \max_{w \in \mathrm{Y} ( y(0) | v_0 )} \langle w , D^2 F(0,y(0)) (1,x) \rangle , $$
	for all $x \in \R^m$. In particular, if $D^2 F(0,y(0)) = 0$, then the above term vanishes. In that particular case, the convex minimization problem~$(\mathcal{M}'_3(0))$ has a quadratic cost with linear inequality/equality constraints (see Remark~\ref{remfin}), while the original parameterized convex minimization problem~$(\mathcal{M}_3(t))$ has a nonlinear cost with nonlinear inequality constraints.
\end{remark}

\section{Additional comments}\label{sec6}

We conclude this paper with some additional comments about the choice of Formula~\eqref{form} and about the twice epi-differentiability of the conjugate function.

\subsection{Comments on the choice of Formula~\eqref{form}}\label{secform}
In the whole paragraph we consider $f  \in \Gamma_0(\cdot,\H)$, $x \in f^{-1} (\cdot , \R)$ and $v \in \partial f(0,x)$. 

\medskip

As mentioned at the beginning of Section~\ref{secmainresults}, the whole paper is based on Attouch's theorems (see Theorems~\ref{thmattouch} and \ref{thmattouch2}) and on Proposition~\ref{prop12}. One can easily see that these results are totally independent of the definition of $\Delta^2_\t f(x|v)$, provided that $\Delta^2_\t f(x|v) \in \Gamma_0(\H)$ and $ \partial ( \Delta^2_\t f (x | v) ) = \Delta_\t (\partial f) (x | v)$. 

\medskip

As a consequence it is clear that one could easily adapt the whole paper with a different definition of $\Delta^2_\t f(x|v)$, as long as $\Delta^2_\t f(x|v) \in \Gamma_0(\H)$ and $ \partial ( \Delta^2_\t f (x | v) ) = \Delta_\t (\partial f) (x | v)$. For example, instead of Formula~\eqref{form}, one could consider
\begin{equation}\label{form2}
\Delta^2_\t f (x | v)(w) := \dfrac{f (\t,x+\t w) - f(0,x) - \t \langle v , w \rangle}{ \t^2 },
\end{equation}
or
\begin{equation}\label{form3}
\Delta^2_\t f (x | v)(w) := \dfrac{f (\t,x+\t w) - f(\t,x) - \t \langle v(\t) , w \rangle}{ \t^2 },
\end{equation}
where $v(t) \in \partial f(t,x)$ for all $t \geq 0$. This is the reason why we think that it is important to justify our choice of Formula~\eqref{form}. Actually this choice was natural and has prevailed with respect to our initial motivation and followed from the case where $f$ is smooth.

\medskip

Indeed let us assume that $f$ is smooth on $\R^+ \times \H$ and let $y, z : \R_+ \to \H$ be two given functions differentiable at $t=0$ that are related by the expression
$$ y(t) = \prox_{f(t,\cdot)} (z(t)) = \prox_{\I, f(t,\cdot)} (z(t)) , $$
for all $t \geq 0$. We deduce that $y(t) + \nabla_x f(t,y(t)) = z(t)$ for all $t \geq 0$ and thus $y'(0)+ \nabla^2_{tx} f(0,y(0)) + \nabla^2_{xx} f(0,y(0))(y'(0)) = z'(0)$. Finally we obtain that $ y'(0) = ( \I + \partial \varphi_{y(0)} )^{-1} (z'(0)) = \prox_{\varphi_{y(0)}} (z'(0)) $ where $\varphi_x : \H \to \R$ is defined by
$$ \varphi_x (w) := \dfrac{1}{2} \left\langle \nabla^2_{xx} f (0,x)(w), w \right\rangle + \left\langle \nabla^2_{tx} f(0,x), w \right\rangle , $$
for all $w \in \H$. Hence it was natural to choose a general expression of $\Delta^2_\t f(x|v)$ that converges pointwise on $\H$ to $\varphi_x$ (up to an additive constant) whenever $f$ is smooth. Note that Formula~\eqref{form} does (see Remark~\ref{rempointwise}).

\medskip

In contrast, if $f$ is smooth and if $\Delta^2_\t f(x|v)(w)$ is defined as in~\eqref{form2}, then $\Delta^2_\t f(x|v)(w)$ does not converge if $\nabla_t f(0,x) \neq 0$. Similarly, if $f$ is smooth and if $\Delta^2_\t f(x|v)(w)$ is defined as in~\eqref{form3}, then $\Delta^2_\t f(x|v)(w)$ converges to $\frac{1}{2} \langle \nabla^2_{xx} f (0,x)(w), w \rangle$ that is different (even up to an additive constant) from $\varphi_x (w)$ if $\nabla^2_{tx} f(0,x) \neq 0$.

%

\subsection{Comments on the twice epi-differentiability of the conjugate function}\label{secdual}
In the $t$-independent framework, recall that the classical \textit{Fenchel conjugate} $f^* \in \Gamma_0(\H)$ of a function $f \in \Gamma_0(\H)$ is defined by
$$ f^* (v) := \sup_{x \in \H} \{ \langle v , x \rangle - f(x) \} , $$
for all $v \in \H$. Recall the following result stated by U.~Mosco in \cite[Theorem~1]{Mosco}. 

\begin{theorem}\label{thmmosco}
	Let $(f_\t)_{\t > 0}$ be a parameterized family of functions in $\Gamma_0(\H)$ and $f \in \Gamma_0(\H)$. Then, $f = \elim f_\t$ if and only if $f^* = \elim f^*_\t$.
\end{theorem}

It can be easily proved that $f^*(v) + f(x) \geq \langle v , x \rangle$ for all $(x,v) \in \H \times \H$, with equality if and only if $v \in \partial f (x)$. From elementary calculus rules for conjugates, for some $x \in f^{-1}(\R)$ and $v \in \partial f(x)$, one can get that the conjugate of $\Delta^2_\t f(x|v) \in \Gamma_0(\H)$ is given by
$$ \Big( \Delta^2_\t f(x|v) \Big)^* = \Delta^2_\t f^* (v|x), $$
for all $\t > 0$. As a consequence, we obtain from Theorem~\ref{thmmosco} that $f$ is twice epi-differentiable at $x$ for $v$ if and only if $f^*$ is twice epi-differentiable at $v$ for $x$. In that case, $d^2_e f(x|v)$ and $d^2_e f^* (v|x)$ both belong to $\Gamma_0(\H)$ from Proposition~\ref{proptinde} and are mutually conjugate, that is, $(d^2_e f(x|v))^* = d^2_e f^* (v|x)$. We refer to \cite[Theorem~13.21 p.604]{RockWets} for details.

\medskip

In the $t$-dependent framework, for the simplicity of notations, we denote by $f^* : \R_+ \times \H \to \R \cup \{ +\infty \}$ the conjugate of any function $f \in \Gamma_0(\cdot,\H)$ defined by 
$$ f^*(t,v) := (f(t,\cdot))^*(v) $$ 
for all $(t,v) \in \R_+ \times \H$. Our aim in this section is to rely the twice epi-differentiability introduced in this paper (see Definition~\ref{deftwiceepidiff}) of $f$ with the twice epi-differentiability of $f^*$ as in the $t$-independent setting. However, as in Section~\ref{secnewassumption}, we will see that the situation is more complicated in the $t$-dependent framework and requires additional assumptions.

\medskip

Let $f \in \Gamma_0(\cdot,\H)$, $x \in f^{-1}(\cdot,\R)$ and $v \in \partial f(0,x)$. We introduce the nonnegative function 
$$ \fonction{\Phi^f_{(x|v)}}{\R_+}{\R_+}{t}{\Phi^f_{(x|v)} (t) := f^*(t,v)+f(t,x)-\langle v,x \rangle.} $$
Note that $\Phi^f_{(x|v)} (0) = 0$. In the particular case where $f$ is $t$-independent, it holds that $\Phi^f_{(x|v)} (t) = 0$ for all $t \geq 0$. Imitating the above $t$-independent framework, it can be proved that
$$ \Big( \Delta^2_\t f(x|v) \Big)^* = \Delta^2_\t f^* (v|x) + \dfrac{\Phi^f_{(x|v)} (\t) }{\t^2} ,$$
for all $\t > 0$. We deduce the following proposition.

\begin{proposition}
	Let $f \in \Gamma_0(\cdot,\H)$, $x \in f^{-1}(\cdot,\R)$ and $v \in \partial f(0,x)$. If $\Phi^f_{(x|v)}$ is twice differentiable at $t=0$ with $(\Phi^f_{(x|v)})'(0)=0$, then $f$ is twice epi-differentiable at $x$ for $v$ with $d^2_e f(x|v) \in \Gamma_0(\H)$ if and only if $f^*$ is twice epi-differentiable at $v$ for $x$ with $d^2_e f^*(v|x) \in \Gamma_0(\H)$. In that case
	$$ (d^2_e f(x|v))^* = d^2_e f^* (v|x) + \frac{1}{2} (\Phi^f_{(x|v)})''(0) .$$
\end{proposition}

\end{document}